\documentclass[twoside]{amsart}

\usepackage[a4paper,top=2.54cm,bottom=2.54cm,left=3.17cm,right=3.17cm,%
            includehead,includefoot]{geometry}

\usepackage{amsmath,amssymb,amsfonts,amsthm}
\usepackage{graphicx}
\usepackage{subfigure}
\usepackage{float}
\usepackage{comment}
\usepackage{xcolor}
\usepackage[numbers,square,sort&compress]{natbib}
\usepackage{hyperref}
\usepackage{mathrsfs}  \hypersetup{colorlinks,citecolor=blue,linkcolor=blue,breaklinks=true}
\usepackage{booktabs}
\usepackage{colortbl}
\usepackage{caption}
\usepackage{enumitem}
\usepackage{epstopdf}
\usepackage{algorithm}
\usepackage{algpseudocode}
\usepackage{array}
\usepackage{bbding}
\usepackage{fancyhdr} 
\usepackage{fancyvrb} 
\usepackage{longtable}
\usepackage{listings}
\usepackage{rotating,rotfloat} 
\usepackage{yhmath} 
\usepackage{amsmath}
\numberwithin{equation}{section}
\usepackage{indentfirst}
\usepackage[mathscr]{euscript}
\usepackage{fancyhdr}
\usepackage{titlesec}

\titleformat{\section}
{\normalfont\Large\bfseries\MakeUppercase}
{\thesection}{1em}{}

\titleformat{\subsection}
{\normalfont\large\bfseries}
{\thesubsection}{1em}{}

\theoremstyle{definition}
\newtheorem{theorem}{Theorem}[section]
\newtheorem{lemma}[theorem]{Lemma}
\newtheorem{corollary}[theorem]{Corollary}
\newtheorem{proposition}[theorem]{Proposition}

\newtheorem*{mainthm}{Main Theorem}
\numberwithin{equation}{section}

\newcommand{\bbm}{\begin{bmatrix}}
\newcommand{\ebm}{\end{bmatrix}}

\title{Differential operator approach to the quantum howe duality of type AIII }
\author{Qianqi Wang}
\address{
School of Mathematical Sciences,
East China Normal University,
Shanghai 200241, China
}
\email{51275500030@stu.ecnu.edu.cn}

\begin{document}

\maketitle
\begin{abstract}
    The quantum Howe duality for iquantum groups of type AIII was established via two approaches, based on partial flag varieties and quantum coordinate coalgebras respectively. In this paper, we provide a new construction of this quantum Howe duality via quantum differential operators.
\end{abstract}
\section{Introduction}
\subsection{Background}
The classical Howe duality, involving commuting actions of a pair of Lie groups or Lie algebras on the same representation space, provides a representation theoretic treatment for classical invariant theory \cite{Ho89}. In the case of general linear groups and Lie algebras, Howe duality is closely related to Schur duality and the first fundamental theorem (FFT).

For the general linear case, R. Zhang \cite{Z02} established the Howe duality between the quantum groups $\mathbf{U}_v(\mathfrak{gl}_m)$ and $\mathbf{U}_v(\mathfrak{gl}_n)$ by introducing quantum coordinate algebras. More precisely, Zhang constructed a noncommutative analogue of the symmetric algebra on which $\mathbf{U}_v(\mathfrak{gl}_m)$ and $\mathbf{U}_v(\mathfrak{gl}_n)$ act. 
The Howe duality for quantum general linear supergroups was first established by Y. Zhang \cite{Zh20} via quantum coordinate superalgebras. This construction also leads to the first and second fundamental theorems of invariant theory for quantum general linear supergroups.
Luo, Yu and Zhou \cite{LYZ26} provided two further constructions of the Howe duality between quantum general linear supergroups, based on quantum differential operators and the Beilinson–Lusztig–MacPherson realization, respectively.

The differential operator approach realizes abstract algebraic actions as explicitly computable operator actions, providing an effective tool for the construction and study of various Howe dualities.
Futorny, K\v{r}i\v{z}ka and Zhang described the quantum Howe dualities \((\mathbf{U}_{v^{2}}(\mathfrak{sl}_{2}),\mathbf{U}_v(\mathfrak{sl}_{2}))\) and \((\mathbf{U}_v(\mathfrak{sl}_{2}),\mathbf{U}_v(\mathfrak{sl}_{n}))\) by means of quantum differential operators on complex vector spaces in \cite{FKZ19}. Moreover, they characterized the spaces of harmonic polynomials as the kernels of certain differential operators and established Fischer decompositions of the corresponding polynomial spaces.
Bodish and Tubbenhauer \cite{BT25} established a non-semisimple version of the Howe duality for the pair$ (\mathbf{U}_v(\mathfrak{sp}_{2n}), \mathbf{U}_v(\mathfrak{sl}_2))$ acting on the type C quantum exterior algebra. They further provided a differential operator description of this Howe duality. 

In \cite{Let99}, Letzter introduced the notion of quantum symmetric pairs. In this framework, an iquantum group $\mathbf{U}^\imath$ is a coideal subalgebra of a quantum group $\mathbf{U}_v(\mathfrak g)$ such that $(\mathbf{U}_v(\mathfrak g),\mathbf{U}^\imath)$ forms a quantum symmetric pair. The classification of quantum symmetric pairs and hence of iquantum groups can be described by Satake diagrams. In \cite{BW18}, Bao and Wang initiated the iprogram by applying iquantum groups
associated with Satake diagrams of type AIII to the reformulation of
Kazhdan--Lusztig theory of type B/C. This demonstrated that many constructions
and results for ordinary quantum groups can be extended naturally to the
iquantum group setting.

Following the development of iquantum groups, quantum Howe dualities for pairs of iquantum groups have also been established. In \cite{LX22}, Luo and Xu established quantum Howe duality of type AIII by two different approaches based on partial flag varieties and quantum coordinate coalgebras respectively. Bao and Wang \cite{BW18} established quantum Schur duality between an iquantum group of type AIII and the Hecke algebra of type B on tensor space. Building on these results, in \cite{Zha26}, by relating the relative braid group action to the Hecke algebra action of type B, W. Zhang showed that the quantum Schur duality can be recovered from the quantum Howe duality of  type AIII.
\subsection{Main results}

Our main result is a new construction of quantum Howe duality of type AIII via quantum differential operators. 

Denote $\mathbb{A}=\mathbb{K}(v)$, where $\mathbb{K}$ is an arbitrary field of characteristic $0$ and $v$ is an indeterminate. Let $m,n$ be non-negative integers and $N=2n+1$, $M=2m+1$.  Let $V$ be a $N$-dimensional $\mathbb{A}$-vector space and denote the symmetric algebra of $V$ by $S(V)$. Then, $S(V)$ is isomorphic to the polynomial algebra $\mathbb{A}[X_{-n},X_{-n+1},...,X_n]$. We write the standard basis elements of $S(V)$ as
\begin{equation*}
    \begin{aligned}
        X^{(\mathbf a)}:=X_{-n}^{a_{-n}}\cdots X_{n}^{a_n},
    \end{aligned}
\end{equation*}
where $\mathbf a=(a_{-n},\cdots,a_n)\in \mathbb{N}^{N}$.
We write the standard basis elements of $S(V)^{\otimes M}$ as
\begin{equation*}
    \begin{aligned}
        X^{[A]}=X^{[r_{-m}]}\otimes\cdots \otimes X^{[r_{m}]},
    \end{aligned}
\end{equation*}
where $A\in \mathrm{Mat}_{M\times N}{(\mathbb{N})}$ and $r_i$ is the $(i+m+1)$-row vector of $A$.
Following \cite[Section 4.1]{LX22}, we introduce the following index sets:
\begin{equation*}
    \begin{aligned}
        \Xi_{m|n,d}^{\jmath\jmath}
        :=&\{
        A=(a_{ij})_{-m\le i\le m,\,-n\le j\le n}
        \in\mathrm{Mat}_{M\times N}(\mathbb N)
        \mid
        a_{ij}=a_{-i,-j},\
        \sum_{i=-m}^{m}\sum_{j=-n}^{n}a_{ij}=2d+1
        \},\\
        \Xi_{m|n,d}^{\imath\jmath}
        :=&\{
        A\in\Xi_{m|n,d}^{\jmath\jmath}
        \mid
        a_{0j}=0\ (j\ne0),\ a_{00}=1
        \},\\
        \Xi_{m|n,d}^{\jmath\imath}
        :=&\{
        A\in\Xi_{m|n,d}^{\jmath\jmath}
        \mid
        a_{i0}=0\ (i\ne0),\ a_{00}=1
        \},\\
        \Xi_{m|n,d}^{\imath\imath}
        :=&\{
        A\in\Xi_{m|n,d}^{\jmath\jmath}
        \mid
        a_{i0}=a_{0j}=0\ (i,j\ne0),\
        a_{00}=1
        \}.       
    \end{aligned}
\end{equation*}
We denote $ \Xi_{m|n}^{\mathfrak{b}\mathfrak{c}}
        :=\bigsqcup_{d\geq0}\Xi_{m|n,d}^{\mathfrak{b}\mathfrak{c}}$, where $\mathfrak{b},\mathfrak{c}\in\{\imath,\jmath\}$. Let $\mathscr{S}^{\mathfrak{b}\mathfrak{c}}_{m|n,d}$ be the subspace of $S(V)^{\otimes M}$ with $\mathbb{A}$-basis $\{X^{[A]}\in S(V)^{\otimes M} \mid A\in \Xi^{\mathfrak{b}\mathfrak{c}}_{m|n,d}\}$ and
        \begin{equation*}
            \begin{aligned}
\mathscr{S}^{\mathfrak{b}\mathfrak{c}}_{m|n}:=\bigoplus_{d=0}^{\infty}\mathscr{S}^{\mathfrak{b}\mathfrak{c}}_{m|n,d}.
            \end{aligned}
        \end{equation*}
We give a $(\mathbf{U}^{\mathfrak{c}}_n,\mathbf{U}^{\mathfrak{b}}_m)$-module structure on $\mathscr{S}^{\mathfrak{b}\mathfrak{c}}_{m|n}$ in Section \ref{section:3}. Moreover, we also give the explicit action formulas.

Let ${}_\mathbb{A}\mathbb{T}^{\mathfrak{c}\mathfrak{b}}_{n|m,d}$ be the Fock space consisting of $\mathrm{SO}_{2d+1}$-invariant functions on the product of two partial flag varieties of type B and $\mathcal{V}^{\mathfrak{b}\mathfrak{c}}_{m|n,d}$ be the subspace of a quantum coordinate coalgebra with  $\mathbb{A}$-basis $\{ \tilde{t}^{(A)}\mid A\in \Xi^{\mathfrak{b}\mathfrak{c}}_{m|n,d}\}$. Denote $\mathcal{V}^{\mathfrak{b}\mathfrak{c}}_{m|n}:=\bigoplus_{d=0}^{\infty}\mathcal{V}^{\mathfrak{b}\mathfrak{c}}_{m|n,d}$ and $_\mathbb{A}\mathbb{T}^{\mathfrak{c}\mathfrak{b}}_{n|m}:=\bigoplus_{d=0}^{\infty} {}_\mathbb{A}\mathbb{T}^{\mathfrak{c}\mathfrak{b}}_{n|m,d}  $. The $(\mathbf{U}^{c}_n,\mathbf{U}^{b}_m)$-bimodules structures of $\mathcal{V}^{\mathfrak{b}\mathfrak{c}}_{m|n}$ and ${}_\mathbb{A}\mathbb{T}^{\mathfrak{c}\mathfrak{b}}_{n|m}$ were introduced in \cite{LX22} and they are isomorphic as $(\mathbf{U}^{c}_n,\mathbf{U}^{b}_m)$-bimodules (see \cite[Theorem 5.7]{LX22}).

\begin{mainthm}
     As $(\mathbf{U}^{c}_n,\mathbf{U}^{b}_m)$-bimodules,
    \begin{equation*}
        \begin{aligned} \mathscr{S}^{\mathfrak{b}\mathfrak{c}}_{m|n,d}\cong\mathcal{V}^{\mathfrak{b}\mathfrak{c}}_{m|n,d},\quad \mathscr{S}^{\mathfrak{b}\mathfrak{c}}_{m|n}\cong\mathcal{V}^{\mathfrak{b}\mathfrak{c}}_{m|n}:X^{[A]}\mapsto \tilde{t}^{(A)}.
        \end{aligned}
\end{equation*}
Consequently, $\mathscr{S}^{\mathfrak{b}\mathfrak{c}}_{m|n,d}\cong{}_\mathbb{A}\mathbb{T}^{\mathfrak{c}\mathfrak{b}}_{n|m,d}$ as $(\mathbf{U}^{c}_n,\mathbf{U}^{b}_m)$-bimodules and the actions
\begin{equation*}
    \begin{aligned}
        \mathbf U_n^{\mathfrak c}
        \curvearrowright
        \mathscr{S}_{m|n,d}^{\mathfrak b\mathfrak c}
        \curvearrowleft
        \mathbf U_m^{\mathfrak b},
        \qquad
        \mathfrak b,\mathfrak c\in\{\imath,\jmath\}
    \end{aligned}
\end{equation*}
satisfy the double centralizer property.
\end{mainthm}
\subsection{Organization}

The paper is organized as follows. In Section \ref{section:preliminary}, we recall the definitions and basic properties of quantum general linear groups $ \mathbf{U}_v(\mathfrak{gl}_n)$ and the $\imath$quantum groups $\mathbf U_n^{\jmath}$ and $\mathbf U_n^{\imath}$. In Section \ref{section:3}, we construct the $(\mathbf U_n^{c},\mathbf U_m^{b})$-bimodule structures on $S_{m|n}^{\mathfrak b\mathfrak c}$ for $\mathfrak b,\mathfrak c\in\{\imath,\jmath\}$. More precisely, in Sections \ref{subsection: left Uj action on Sbj} and \ref{subsection: left Ui action on Sbi}, we construct the left actions via quantum differential operators and derive explicit action formulas. In Section \ref{subsection: right Ub action Sbc}, we obtain the right actions via matrix transposition and algebra anti-automorphisms of iquantum groups, and also give the explicit action formulas. Finally, comparing these action formulas with those in \cite{LX22}[Proposition 5.5], we identify our construction with the geometric and quantum coordinate coalgebra constructions and establish the double centralizer property.

\subsection*{Acknowledgment}
We thank Li Luo for suggesting the problem studied in this paper and for providing valuable advice. We thank Xirui Yu for many stimulating discussions and suggestions.

\section{Preliminary}\label{section:preliminary}

\subsection{Quantum general linear groups \(\mathbf{U}_v(\mathfrak{gl}_n)\)}
The quantum general linear groups $\mathbf{U}_v(\mathfrak{gl}_n)$ is defined as the algebra over $\mathbb{A}$, generated by $E_i,F_i(i=1,2,...,n-1)$ and $D_a^{\pm1}(a=1,2,...,n)$, subject to the following relations:
\begin{equation*}
    \begin{aligned}
    &D_aD_a^{-1}=D_a^{-1}D_a=1,
    \quad
    D_a^{\pm1}D_b^{\pm1}=D_b^{\pm1}D_a^{\pm1},
    \\
    &D_aE_iD_a^{-1}
    =
    v^{\delta_{a,i}-\delta_{a,i+1}}E_i,
    \quad
    D_aF_iD_a^{-1}
    =
    v^{\delta_{a,i+1}-\delta_{a,i}}F_i,
    \\
    &E_iF_j-F_jE_i
    =
    \delta_{ij}\frac{K_i-K_i^{-1}}{v-v^{-1}},
    \quad \mathrm{where} \ 
    K_i=D_iD_{i+1}^{-1},
    \\
    &E_iE_j=E_jE_i,
    \quad
    F_iF_j=F_jF_i,
    \quad
     \mathrm{for}\ |i-j|>1,
    \\
    &E_i^2E_j+E_jE_i^2
    =
    (v+v^{-1})E_iE_jE_i,
    \quad
    \mathrm{for}\ |i-j|=1,
    \\
    &F_i^2F_j+F_jF_i^2
    =
    (v+v^{-1})F_iF_jF_i,
    \quad
    \mathrm{for}\ |i-j|=1.
    \end{aligned}
\end{equation*}
The Hopf algebra structure on $\mathbf{U}_v(\mathfrak{gl}_n)$ is given by the comultiplication $\Delta$, the counit $\varepsilon$ and the antipode $S$ as follows:
\begin{equation*}
\begin{aligned}
&\Delta(E_i)=E_i\otimes K_i^{-1}+1\otimes E_i,
\quad
\Delta(F_i)=F_i\otimes 1+K_i\otimes F_i,
\quad\Delta(D_a)=D_a\otimes D_a,
\\
&\varepsilon(E_i)=0,
\quad
\varepsilon(F_i)=0,
\quad
\varepsilon(D_a)=1,
\\
&S(E_i)=-E_iK_i,
\quad
S(F_i)=-K_i^{-1}F_i,
\quad
S(D_a)=D_a^{-1}.
\end{aligned}
\end{equation*}
It's easy to check that there is an algebra anti-automorphism $\omega$ of $\mathbf{U}_v(\mathfrak{gl}_n)$ such that
\begin{equation}\label{equation: an anti-automorphsim of quantum group}
    \begin{aligned}
        \omega(E_i)=vF_iK_i^{-1},\quad \omega(F_i)=vE_iK_i,\quad \omega(D_a)=D_a.
    \end{aligned}
\end{equation}
\subsection{iquantum groups $\mathbf{U}^{\jmath}_n$ and $\mathbf{U}^{\imath}_n$}
Denote $N=2n+1$ and
$\mathbb{I}_N=\{-n,-n+1,\ldots,n-1,n\}$, $
\mathbb{I}_N^{\frac{1}{2}}=\{-n+\frac{1}{2},-n+\frac{3}{2},\ldots,n-\frac{1}{2}\}$. We take a shift by $-n-\frac{1}{2}$(respectively, $-n-1$) on the index of the generators $E_i,F_i$ (respectively, $D_j$) of $\mathbf{U}_v(\mathfrak{gl}_N)$. The iquantum groups $\mathbf{U}^{\jmath}_n$ is defined as the $\mathbb{A}$-subalgebra of $\mathbf{U}_v(\mathfrak{gl}_N)$ generated by
\begin{equation}\label{equation: generators of type j}
    \begin{aligned}
        &e_i=E_{i+\frac{1}{2}}+D_i^{-1}D_{i+1}F_{-(i+\frac{1}{2})},\quad f_i=E_{-(i+\frac{1}{2})}+F_{i+\frac{1}{2}}D_{-(i+1)}^{-1}D_{-i}, \quad(0\leq i<n),\\
        &d_0^{\pm1}=D_0^{\pm1}, \quad  d_j^{\pm1}=(D_jD_{-j})^{\pm1} ,\quad (0< j\leq n).
    \end{aligned}
\end{equation}
where $K_i=D_{i-\frac{1}{2}}D_{i+\frac{1}{2}}^{-1}(i\in \mathbb{I}_N^{\frac{1}{2}})$. We have
\begin{equation*}
\begin{aligned}
&\Delta(e_i)=e_i\otimes K_{i+\frac{1}{2}}^{-1}
+1\otimes E_{i+\frac{1}{2}}
+k_i^{-1}\otimes K_{i+\frac{1}{2}}^{-1}F_{-\left(i+\frac{1}{2}\right)},
\\
&\Delta(f_i)=f_i\otimes K_{-\left(i+\frac{1}{2}\right)}^{-1}
+1\otimes E_{-\left(i+\frac{1}{2}\right)}
+k_i\otimes F_{i+\frac{1}{2}}K_{-\left(i+\frac{1}{2}\right)}^{-1},
\\
&\Delta(d_j)=d_j\otimes d_j,
\end{aligned}
\end{equation*}
where $k_i=K_{i+\frac{1}{2}}K_{-i-\frac{1}{2}}^{-1}=d_i^{1+\delta_{i,0}}d_{i+1}^{-1},(0\leq i<n)$. So $\Delta(\mathbf{U}^{\jmath}_n)\subseteq \mathbf{U}^{\jmath}_n \otimes \mathbf{U}_v(\mathfrak{gl}_N)$.

Denote
$\mathbb{I}_{2n}
=
\{-n+\frac{1}{2},-n+\frac{3}{2},\ldots,n-\frac{1}{2}\},
\mathbb{I}_{2n}^{\frac{1}{2}}
=
\{-n+1,-n+2,\ldots,n-1\}$. We take a shift by $-n$ (respectively, $-n-\frac{1}{2}$) on the index of the generators $E_i,F_i$ (respectively, $D_j$) of $\mathbf{U}_v(\mathfrak{gl}_{N-1})$. The iquantum groups $\mathbf{U}^{\imath}_n$ is defined as the $\mathbb{A}$-subalgebra of $\mathbf{U}_v(\mathfrak{gl}_{N-1})$ generated by
\begin{equation}\label{equation: generators of type i}
    \begin{aligned}
        e_i&=E_{i}+D_{i-\frac{1}{2}}^{-1}D_{i+\frac{1}{2}}F_{-i},\quad f_i=E_{-i}+F_iD_{-i-\frac{1}{2}}^{-1}D_{-i+\frac{1}{2}} \ (0<i<n),\\
        t_0&=E_0+vF_0D_{-\frac{1}{2}}^{-1}D_{\frac{1}{2}}+D_{-\frac{1}{2}}^{-1}D_{\frac{1}{2}},\quad d_j^{\pm1}=(D_{j-\frac{1}{2}}D_{-j+\frac{1}{2}})^{\pm1}\ (0<j\leq n),
    \end{aligned}
\end{equation}
where $K_i=D_{i-\frac{1}{2}}D_{i+\frac{1}{2}}^{-1},(i\in \mathbb{I}_{2n}^{\frac{1}{2}})$. We have
\begin{equation*}
\begin{aligned}
&\Delta(e_i)
=
e_i\otimes K_i^{-1}
+1\otimes E_i
+k_i^{-1}\otimes K_i^{-1}F_{-i},~\Delta(f_i)
=
f_i\otimes K_{-i}^{-1}
+1\otimes E_{-i}
+k_i\otimes F_iK_{-i}^{-1},
\\
&\Delta(t_0)
=
t_0\otimes K_0^{-1}
+1\otimes\left(E_0+vF_0K_0^{-1}\right),~\Delta(d_j)
=
d_j\otimes d_j,
\end{aligned}
\end{equation*}
where $k_i=K_iK_{-i}^{-1}=d_id_{i+1}^{-1},(0<i<n)$.
So $\Delta(\mathbf{U}^{\imath}_n)\subseteq \mathbf{U}^{\imath}_n \otimes \mathbf{U}_v(\mathfrak{gl}_{N-1})$.

\section{$(\mathbf{U}^{c}_n,\mathbf{U}^{b}_m)$-bimodules structures of $\mathscr{S}^{\mathfrak{b} \jmath}_{m|n}$}\label{section:3}
\subsection{The left action of $\mathbf{U}^{\jmath}_n$ on $\mathscr{S}^{\mathfrak{b} \jmath}_{m|n}$}\label{subsection: left Uj action on Sbj}
 Let $\epsilon_i\in\mathbb{N}^{N} $ be the row vector whose $i$-th entry is $1$ and all other entries are zero.  We equip $S(V)$ with the $\mathbf{U}_v(\mathfrak{gl}_N)$-module structure realized by quantum differential operators in \cite[Example 2.2(1)]{DZ20}. Explicitly, the generators act on the monomial basis as follows:
\begin{equation}\label{equation: formula of the action of UqglN on SV}
E_{i+\frac{1}{2}}\cdot X^{(\mathbf{a})}=[a_{i+1}]X^{(\mathbf{a}+\epsilon_{i}-\epsilon_{i+1})},~F_{i+\frac{1}{2}}\cdot X^{(\mathbf{a})}=[a_{i}]X^{(\mathbf{a}-\epsilon_{i}+\epsilon_{i+1})},~D_h\cdot X^{(\mathbf{a})}=v^{a_h}X^{(\mathbf{a})},
\end{equation}
where $-n\leq h\leq n$, $-n\leq i\leq n-1$ and $[n]=\frac{v^{n}-v^{-n}}{v-v^{-1}}$ is the quantum integer. Since iquantum group $\mathbf{U}^{\jmath}_n$ is a subalgebra of $\mathbf{U}_v(\mathfrak{gl}_{N})$,
we can view $S(V)$ as a $\mathbf{U}^{\jmath}_n$-module. We now proceed to construct a quotient $\mathbf{U}^{\jmath}_n$-module of $S(V)$.
For $\mathbf{k}=(k_1,\cdots,k_n)$, $\mathbf{b}=(b_1,\cdots ,b_n)\in \mathbb{N}^n$, $t\in \mathbb{N}$ we rewrite the elements of $S(V)$ as
\begin{equation*}
    \begin{aligned}
        X_{\mathbf{k},t,\mathbf{b}}=X_{-n}^{k_n}\cdots X_{-1}^{k_1}X_{0}^{t}X_{1}^{b_1}\cdots X_{n}^{b_n}.
    \end{aligned}
\end{equation*}
We set
\begin{equation*}
    \begin{aligned}
        &K_{\mathbf{k}}=k_1+\cdots+k_n,\quad
        Q_{\mathbf{k},t,\mathbf{b}}=\frac{K_{\mathbf{k}}(K_{\mathbf{k}}+1)}{2}+tK_{\mathbf{k}}+\sum\limits_{1\leq r<s\leq n} b_rk_s,
        \\
        &C_{\mathbf{k},t,\mathbf{b}}=v^{-Q_{\mathbf{k},t,\mathbf{b}}}.
    \end{aligned}
\end{equation*}
Let $\mathcal{M}^{\jmath}$ be the $\mathbb{A}$-subspace of $S(V)$ defined by
\begin{equation*}
    \begin{aligned}
        \mathcal{M}^{\jmath}=\mathrm{span}_{\mathbb{A}}\{ R_{\mathbf{k},t,\mathbf{b}}=X_{\mathbf{k},t,\mathbf{b}}-C_{\mathbf{k},t,\mathbf{b}}X_{\mathbf{0},t,\mathbf{k}+\mathbf{b}}\in S(V) \mid \mathbf{k},\mathbf{b}\in \mathbb{N} ^n,t\in \mathbb{N}\}.
    \end{aligned}
\end{equation*}
\begin{lemma}
    The constants $C_{\mathbf{k},t,\mathbf{b}}\in \mathbb{A}$ satisfy the following relations:
    \begin{equation}\label{equation: relations of constants Cktb}
        \begin{aligned}
            &C_{\mathbf{k},t,\mathbf{b}+\epsilon_i-\epsilon_{i+1}}=v^{-k_{i+1}} C_{\mathbf{k},t,\mathbf{b}},
            &&C_{\mathbf{k},t,\mathbf{b}-\epsilon_i+\epsilon_{i+1}}=v^{k_{i+1}} C_{\mathbf{k},t,\mathbf{b}},\\
            &C_{\mathbf{k}+\epsilon_i-\epsilon_{i+1},t,\mathbf{b}}=v^{b_i}C_{\mathbf{k},t,\mathbf{b}},
            &&C_{\mathbf{k}-\epsilon_i+\epsilon_{i+1},t,\mathbf{b}}=v^{-b_i}C_{\mathbf{k},t,\mathbf{b}},\\
            &C_{\mathbf{k},t+1,\mathbf{b}-\epsilon_1}=v^{-k_1}C_{\mathbf{k},t,\mathbf{b}}, 
            &&C_{\mathbf{k},t-1,\mathbf{b}+{\epsilon_1}}=v^{k_1}C_{\mathbf{k},t,\mathbf{b}},\\
            &C_{\mathbf{k}-\epsilon_1,t+1,\mathbf{b}}=v^{t+1}C_{\mathbf{k},t,\mathbf{b}},
            &&C_{\mathbf{k}+\epsilon_1,t-1,\mathbf{b}}=v^{-t}C_{\mathbf{k},t,\mathbf{b}}.
        \end{aligned}
    \end{equation}
\end{lemma}
\begin{proof}
    Note that $\sum\limits_{1\leq r<s\leq n} b_rk_s=b_1k_2+(b_1+b_2)k_3+\cdots (b_1+\cdots+b_{n-1})k_n$, it follows from straightforward computations.
\end{proof}
\begin{proposition}\label{prop: Mj is a Uj module}
      The subspace $\mathcal{M}^{\jmath}$ is a $\mathbf{U}^{\jmath}_n$-submodule of $S(V)$.
\end{proposition}
\begin{proof}
    It suffices to show that the actions of generators $e_i,f_i,d_j^{\pm}$ of $\mathbf{U}^{\jmath}_n$ on $R_{\mathbf{k},t,\mathbf{b}}$ remain in $\mathcal{M}^{\jmath}$. For $i>0$, it follows from (\ref{equation: generators of type j}), (\ref{equation: formula of the action of UqglN on SV}), (\ref{equation: relations of constants Cktb}) that 
    \begin{equation*}
        \begin{aligned}
            e_i\cdot X_{\mathbf{k},t,\mathbf{b}}=&[b_{i+1}]X_{\mathbf{k},t,\mathbf{b}+\epsilon_i-\epsilon_{i+1}}+v^{b_{i+1}-b_i}[k_{i+1}]X_{\mathbf{k}+\epsilon_i-\epsilon_{i+1},t,\mathbf{b}}\\
            =&([b_{i+1}]C_{\mathbf{k},t,\mathbf{b}+\epsilon_i-\epsilon_{i+1}}
            +v^{b_{i+1}-b_i}[k_{i+1}]C_{\mathbf{k}+\epsilon_i-\epsilon_{i+1},t,\mathbf{b}})X_{\mathbf{0},t,\mathbf{k}+\mathbf{b}+\epsilon_i-\epsilon_{i+1}}\\
            &+[b_{i+1}]R_{\mathbf{k},t,\mathbf{b}+\epsilon_i-\epsilon_{i+1}}+v^{b_{i+1}-b_i}[k_{i+1}]R_{\mathbf{k}+\epsilon_i-\epsilon_{i+1},t,\mathbf{b}}\\
            =&(v^{-k_{i+1}}[b_{i+1}]+v^{b_{i+1}}[k_{i+1}])C_{\mathbf{k},t,\mathbf{b}}X_{\mathbf{0},t,\mathbf{k}+\mathbf{b}+\epsilon_i-\epsilon_{i+1}}\\
            &+[b_{i+1}]R_{\mathbf{k},t,\mathbf{b}+\epsilon_i-\epsilon_{i+1}}+v^{b_{i+1}-b_i}[k_{i+1}]R_{\mathbf{k}+\epsilon_i-\epsilon_{i+1},t,\mathbf{b}}.
        \end{aligned}
    \end{equation*}
    Note that the quantum integers satisfy the relation $[r+s]=v^{-r}[s]+v^{s}[r]$, we have 
    \begin{equation*}
        \begin{aligned}
            e_i\cdot(C_{\mathbf{k},t,\mathbf{b}}X_{\mathbf{0},t,\mathbf{k}+\mathbf{b}})&=[k_{i+1}+b_{i+1}]C_{\mathbf{k},t,\mathbf{b}}X_{\mathbf{0},t,\mathbf{k}+\mathbf{b}+\epsilon_i-\epsilon_{i+1}}\\
            &=(v^{-k_{i+1}}[b_{i+1}]+v^{b_{i+1}}[k_{i+1}])C_{\mathbf{k},\mathbf{b}}X_{\mathbf{0},t,\mathbf{k}+\mathbf{b}+\epsilon_i-\epsilon_{i+1}}.
        \end{aligned}
    \end{equation*}
    Thus,
    \begin{equation*}
        e_i.R_{\mathbf{k},t,\mathbf{b}}=[b_{i+1}]R_{\mathbf{k},t,\mathbf{b}+\epsilon_i-\epsilon_{i+1}}+v^{b_{i+1}-b_i}[k_{i+1}]R_{\mathbf{k}+\epsilon_i-\epsilon_{i+1},t,\mathbf{b}}\in \mathcal{M}^{\jmath}.
    \end{equation*}
    Similarly, we have
    \begin{equation*}
        \begin{aligned}
            f_i\cdot X_{\mathbf{k},t,\mathbf{b}}=&[k_i]X_{\mathbf{k}-\epsilon_i+\epsilon_{i+1},t,\mathbf{b}}+v^{k_i-k_{i+1}}[b_i]X_{\mathbf{k},t,\mathbf{b}-\epsilon_i+\epsilon_{i+1}}\\
            =&(v^{-b_i}[k_i]+v^{k_i}[b_i])C_{\mathbf{k},t,\mathbf{b}}X_{\mathbf{0},t,\mathbf{k}+\mathbf{b}-\epsilon_i
            +\epsilon_{i+1}}
            \\
            &+[k_i]R_{\mathbf{k}-\epsilon_i+\epsilon_{i+1},t,\mathbf{b}}+v^{k_i-k_{i+1}}[b_i]R_{\mathbf{k},t,\mathbf{b}-\epsilon_i+\epsilon_{i+1}},\\
            f_i\cdot(C_{\mathbf{k},t,\mathbf{b}}X_{\mathbf{0},t,\mathbf{k}+\mathbf{b}}&)
            =[k_i+b_i]C_{\mathbf{k},t,\mathbf{b}}X_{\mathbf{0},t,\mathbf{k}+\mathbf{b}-\epsilon_i+\epsilon_{i+1}}=(v^{-b_i}[k_i]+v^{k_i}[b_i])C_{\mathbf{k},t,\mathbf{b}}X_{\mathbf{0},t,\mathbf{k}+\mathbf{b}-\epsilon_i+\epsilon_{i+1}}.
        \end{aligned}
    \end{equation*}
    Thus,
    \begin{equation*}
        \begin{aligned}
            f_i.R_{\mathbf{k},t,\mathbf{b}}=[k_i]R_{\mathbf{k}-\epsilon_i+\epsilon_{i+1},t,\mathbf{b}}+v^{k_i-k_{i+1}}[b_i]R_{\mathbf{k},t,\mathbf{b}-\epsilon_i+\epsilon_{i+1}}\in \mathcal{M}^{\jmath}.
        \end{aligned}
    \end{equation*}
    Next, we compute the actions of \(e_0\) and \(f_0\) on \(R_{\mathbf{k},t,\mathbf{b}}\). Again by (\ref{equation: generators of type j}), (\ref{equation: formula of the action of UqglN on SV}), (\ref{equation: relations of constants Cktb}), we have
    \begin{equation*}
        \begin{aligned}
            e_0\cdot X_{\mathbf{k},t,\mathbf{b}}
            =&[b_1]X_{\mathbf{k},t+1,\mathbf{b}-\epsilon_1}+v^{b_1-t-1}[k_1]X_{\mathbf{k}-\epsilon_1,t+1,\mathbf{b}}\\
            =&([b_1]C_{\mathbf{k},t+1,\mathbf{b}-\epsilon_1}+v^{b_1-t-1}[k_1]C_{\mathbf{k}-\epsilon_1,t+1,\mathbf{b}})X_{\mathbf{0},t+1,\mathbf{b}+\mathbf{k}-\epsilon_1}\\
            &+[b_1]R_{\mathbf{k},t+1,\mathbf{b}-\epsilon_1}+v^{b_1-t-1}[k_1]R_{\mathbf{k}-\epsilon_1,t+1,\mathbf{b}}\\
            =&([b_1]v^{-k_1}+v^{b_1}[k_1]C_{\mathbf{k}-\epsilon_1,t+1,\mathbf{b}})X_{\mathbf{0},t+1,\mathbf{b}+\mathbf{k}-\epsilon_1}\\
            &+[b_1]R_{\mathbf{k},t+1,\mathbf{b}-\epsilon_1}+v^{b_1-t-1}[k_1]R_{\mathbf{k}-\epsilon_1,t+1,\mathbf{b}},
            \\
            e_0\cdot(C_{\mathbf{k},t,\mathbf{b}}X_{\mathbf{0},t,\mathbf{k}+\mathbf{b}}&)
            =[k_1+b_1]C_{\mathbf{k},t,\mathbf{b}}X_{\mathbf{0},t+1,\mathbf{k}+\mathbf{b}-\epsilon_1}=(v^{-k_1}[b_1]+v^{b_1}[k_1])C_{\mathbf{k},t,\mathbf{b}}X_{\mathbf{0},t+1,\mathbf{k}+\mathbf{b}-\epsilon_1}.
        \end{aligned}
    \end{equation*}
    Thus, 
    \begin{equation*}
        \begin{aligned}
            e_0.R_{\mathbf{k},t,\mathbf{b}}=[b_1]R_{\mathbf{k},t+1,\mathbf{b}-\epsilon_1}+v^{b_1-t-1}[k_1]R_{\mathbf{k}-\epsilon_1,t+1,\mathbf{b}}\in \mathcal{M}^{\jmath}.
        \end{aligned}
    \end{equation*}
    Similarly, we have
    \begin{equation*}
        \begin{aligned}
            f_0\cdot X_{\mathbf{k},t,\mathbf{b}}
            =&[t]X_{\mathbf{k}+\epsilon_1,t-1,\mathbf{b}}+v^{t-k_1}[t]X_{\mathbf{k},t-1,\mathbf{b}+\epsilon_1}\\
            =&(C_{\mathbf{k}+\epsilon_1,t-1,\mathbf{b}}+v^{t-k_1}C_{\mathbf{k},t-1,\mathbf{b}+\epsilon_1})[t]X_{\mathbf{0},t-1,\mathbf{b}+\mathbf{k}-\epsilon_1}\\
            &+[t]R_{\mathbf{k}+\epsilon_1,t-1,\mathbf{b}}+v^{t-k_1}[t]R_{\mathbf{k},t-1,\mathbf{b}+\epsilon_1}\\
            =&(v^{-t}+v^t)[t]C_{\mathbf{k},t,\mathbf{b}}X_{\mathbf{0},t-1,\mathbf{b}+\mathbf{k}-\epsilon_1}\\
            &+[t]R_{\mathbf{k}+\epsilon_1,t-1,\mathbf{b}}+v^{t-k_1}[t]R_{\mathbf{k},t-1,\mathbf{b}+\epsilon_1},
            \\
            f_0\cdot(C_{\mathbf{k},t,\mathbf{b}}X_{\mathbf{0},t,\mathbf{k}+\mathbf{b}}&)
            =[t]C_{\mathbf{k},t,\mathbf{b}}X_{\epsilon_1,t-1,\mathbf{k}+\mathbf{b}}+v^t[t]C_{\mathbf{k},t,\mathbf{b}}X_{\mathbf{0},t-1,\mathbf{k}+\mathbf{b}+\epsilon_1}\\
            &~=(C_{\epsilon_1,t-1,\mathbf{b}}+v^{t})[t]C_{\mathbf{k},t,\mathbf{b}}X_{\mathbf{0},t-1,\mathbf{k}+\mathbf{b}+\epsilon_1}+[t]C_{\mathbf{k},t,\mathbf{b}}R_{\epsilon_1,t-1,\mathbf{k}+\mathbf{b}}.
        \end{aligned}
    \end{equation*}
    Note that $C_{\epsilon_1,t-1,\mathbf{b}}=v^{-t}$, we have
    \begin{equation*}
        \begin{aligned}
            f_0.R_{\mathbf{k},t,\mathbf{b}}=[t]R_{\mathbf{k}+\epsilon_1,t-1,\mathbf{b}}+v^{t-k_1}[t]R_{\mathbf{k},t-1,\mathbf{b}+\epsilon_1}-[t]C_{\mathbf{k},t,\mathbf{b}}R_{\epsilon_1,t-1,\mathbf{k}+\mathbf{b}}\in \mathcal{M}^{\jmath}.
        \end{aligned}
    \end{equation*}
    The computation for $d_j^{\pm1}$ is easier. We have
    \begin{equation*}
        \begin{aligned}
            d_0^{\pm 1}\cdot X_{\mathbf{k},t,\mathbf{b}}
            =v^{\pm t} X_{\mathbf{k},t,\mathbf{b}}=v^{\pm t}&C_{\mathbf{k},t,\mathbf{b}}X_{\mathbf{0},t,\mathbf{k}+\mathbf{b}}+v^{\pm t}R_{\mathbf{k},t,\mathbf{b}}=d_{0}^{\pm1}(C_{\mathbf{k},t,\mathbf{b}}X_{\mathbf{0},t,\mathbf{k}+\mathbf{b}})+v^{\pm t}R_{\mathbf{k},t,\mathbf{b}},\\
            d_j^{\pm1}\cdot X_{\mathbf{k},t,\mathbf{b}}=v^{\pm(b_j+k_j)}X_{\mathbf{k},t,\mathbf{b}}
            &=v^{\pm(b_j+k_j)}C_{\mathbf{k},t,\mathbf{b}}X_{\mathbf{0},t,\mathbf{k}+\mathbf{b}}+v^{\pm(b_j+k_j)}R_{\mathbf{k},t,\mathbf{b}}
            \\&=d_j^{\pm1}(C_{\mathbf{k},t,\mathbf{b}}X_{\mathbf{0},t,\mathbf{k}+\mathbf{b}})+v^{\pm(b_j+k_j)}R_{\mathbf{k},t,\mathbf{b}},\quad (j>0).
        \end{aligned}
    \end{equation*}
    Thus, $d_0 \cdot R_{\mathbf{k},t,\mathbf{b}}=v^{\pm t} R_{\mathbf{k},t,\mathbf{b}}$, $d_j \cdot R_{\mathbf{k},t,\mathbf{b}}=v^{\pm(b_j+k_j)}R_{\mathbf{k},t,\mathbf{b}},\ (j>0)$.
\end{proof}
Let $S^{\jmath}(V):=S(V)/\mathcal{M}^{\jmath}$, then $S^{\jmath}$ is a quotient $\mathbf{U}^{\jmath}_n$-module of $S(V)$ by the above proposition. By abuse of notation, we also denote by $X^{(\mathbf{a})}$($X_{\mathbf{k},t,\mathbf{b}}$) the elements of $S^{\jmath}(V)$ and always remember the relations $R_{\mathbf{k},t,\mathbf{b}}=0$. It follows from the definition of $\mathcal M^\jmath$ that $\{X_{\mathbf0,t,\mathbf{a}}, t\in\mathbb N, \mathbf{a}\in\mathbb N^n\}$ form an $\mathbb A$-basis of $S^\jmath(V)$.
\begin{corollary}
    For $X_{\mathbf{0},t,\mathbf{a}}\in S^{\jmath}(V)$, $e_i$, $f_i$, $d_j\in\mathbf{U}^{\jmath}_n $, we have
    \begin{equation}\label{equation: the action of type j on X0,t,a}
        \begin{aligned}
         e_0\cdot X_{\mathbf{0},t,\mathbf{a}}&=[a_1]X_{\mathbf{0},t+1,\mathbf{a}-\epsilon_1},\quad &&e_i\cdot X_{\mathbf{0},t,\mathbf{a}}=[a_{i+1}]X_{\mathbf{0},t,\mathbf{a}+\epsilon_i-\epsilon_{i+1}},\ (0<i<n),\\
            f_0\cdot X_{\mathbf{0},t,\mathbf{a}}&=[2t]X_{\mathbf{0},t-1,\mathbf{a}+\epsilon_1},\quad
            &&f_i\cdot X_{\mathbf{0},t,\mathbf{a}}=[a_i]X_{\mathbf{0},t,\mathbf{a}-\epsilon_i+\epsilon_{i+1}},\ (0<i<n),\\
            d_0\cdot X_{\mathbf{0},t,\mathbf{a}}&=v^{t},&&d_j\cdot X_{\mathbf{0},t,\mathbf{a}}=v^{a_j}X_{\mathbf{0},t,\mathbf{a}},\ (0< j\leq n).
        \end{aligned}
    \end{equation}
\end{corollary}
\begin{proof}
    It follows from the proof of Proposition \ref{prop: Mj is a Uj module} directly.
\end{proof}
 We set 
 \begin{equation*}
     \begin{aligned}
         S^{\jmath}_{M}(V):=(S(V))^{\otimes m}\otimes S^{\jmath}(V) \otimes (S(V))^{\otimes m}
     \end{aligned}
 \end{equation*}
 and write the elements of $S^{\jmath}_{M}$ as
$X^{[A]}=X^{(r_{-m})}\otimes\cdots\otimes X^{(r_m)}$,
where $A \in \mathrm{Mat}_{M\times N}{(\mathbb{N})} $ and $r_i$ is the $(i+m+1)$-row vector of $A$. Next, we define the $\mathbf{U}^{\jmath}_n$-module structure on $S^{\jmath}_{M}(V)$. Since $\Delta^m(\mathbf{U}^{\jmath}_n)\subseteq\mathbf{U}^{\jmath}_n\otimes \Delta^{m-1}(\mathbf{U}_v(\mathfrak{gl}_{N}))$ and $S^{\jmath}(V)$(respectively, $S(V)$) is $\mathbf{U}^{\jmath}_n$(respectively, $\mathbf{U}_v(\mathfrak{gl}_{N}))$)module, we can define the module structure via comultiplication. Explicitly, for $x\in \mathbf{U}^{\jmath}_n, X^{[A]}\in S^{\jmath}_{M}(V)$, we define
\begin{equation*}
    \begin{aligned}
        x\cdot X^{[A]}=(1^{\otimes m}\otimes \Delta^m(x))\cdot X^{[A]} .
    \end{aligned}
\end{equation*}

Next, we give the $\mathbf{U}^{\jmath}_n$-module structure  on $\mathscr{S}^{\mathfrak{b}\jmath}_{m|n}$ for $\mathfrak{b}\in \{\imath,\jmath\}$. We use the lexicographical order $<$ on $\mathbb{Z}^2$ given by $(i,j)<(k,l)$ if and only if $i<k$ or $i=k$, $j<l$.
For any $A=(a_{i,j})_{M\times N}\in \Xi^{\mathfrak{b} \jmath}_{m|n}$, we define $a^{\#}_{i,j}=
\begin{cases}
\frac{a_{0,0}-1}{2},&(i,j)=(0,0)\\
a_{i,j},  & (i,j)\neq(0,0).
\end{cases}$ 
Define a map $\rho:\Xi^{\mathfrak{b} \jmath}\to\mathrm{Mat}_{M\times N}{(\mathbb{N})}$
satisfying
\begin{equation*}
    \begin{aligned}
        \rho(A)_{i,j}=0 ,\quad (i,j)<(0,0),\quad \rho(A)_{i,j}=a_{i,j}^{\#} ,\quad (i,j)\geq(0,0).
    \end{aligned}
\end{equation*}
Since $\rho$ is  injective and $\{X^{[A]}\mid A\in \rho (\Xi^{\mathfrak{b} \jmath}_{m|n})\}$ is linearly independent in $S^{\jmath}_{M}$, we obtain a injection:
$\phi :\mathscr{S}^{\mathfrak{b}\jmath}_{m|n}\xrightarrow{} S^{\jmath}_{M},~X^{[A]}\mapsto X^{[\rho(A)]}$.
Then for any $x\in \mathbf{U}^{\jmath}_n$, $X^{[A]}\in \mathscr{S}^{\mathfrak{b}\jmath}_{m|n}$, we define 
\begin{equation*}
    \begin{aligned}
        x\cdot X^{[A]}:=\phi^{-1}( x\cdot\phi(X^{[A]})),
    \end{aligned}
\end{equation*} 
which gives a $ \mathbf{U}^{\jmath}_n$-module structure on $\mathscr{S}^{\mathfrak{b}\jmath}_{m|n}$. It remains to verify that $\mathrm{Im}\phi$ is stable under the action  of $\mathbf{U}^{\jmath}_n$ to ensure that the definition is well-defined.

\begin{lemma}\label{lemma: Imphi is stable under the action of Uj}
Under the above notation,
    $\mathbf{U}^{\jmath}_n\cdot \mathrm{Im}\phi \subseteq \mathrm{Im}\phi$.
\end{lemma}
\begin{proof}
    Note that $\mathrm{Im}\phi=\mathrm{span}\{X^{[A]}\mid A\in \rho (\Xi^{\mathfrak{b} \jmath}_{m|n})\}$.
   By induction on $m$, for $0\leq i<n$, $0\leq j\leq n$, we have
    \begin{equation*}
        \begin{aligned}
            \Delta^{m}(e_i)=&e_i\otimes (K_{i+\frac{1}{2}}^{-1})^{\otimes m}+\sum\limits_{j=1}^m 1^{\otimes j}\otimes E_{i+\frac{1}{2}}\otimes(K_{i+\frac{1}{2}}^{-1})^{\otimes (m-j)}\\
            &+\sum\limits_{j=1}^m(k_{i}^{-1})^{\otimes j}\otimes K_{i+\frac{1}{2}}^{-1}F_{-i-\frac{1}{2}}\otimes (K_{i+\frac{1}{2}}^{-1})^{\otimes (m-j)},\\
            \Delta^{m}(f_i)=&f_i\otimes (K_{-i-\frac{1}{2}}^{-1})^{\otimes m }+\sum\limits_{j=1}^m 1^{\otimes j}\otimes E_{-i-\frac{1}{2}} \otimes(K_{-i-\frac{1}{2}}^{-1})^{\otimes (m-j)}
            \\&+\sum\limits_{j=1}^m(k_{i})^{\otimes j}\otimes F_{i+\frac{1}{2}}K_{-i-\frac{1}{2}}^{-1}\otimes (K_{-i-\frac{1}{2}}^{-1})^{\otimes (m-j)},\\
            \Delta^{m}(d_j)=&(d_j)^{\otimes(m+1)}.
        \end{aligned}
    \end{equation*}
    It follows from (\ref{equation: formula of the action of UqglN on SV}),(\ref{equation: the action of type j on X0,t,a}) that
    \begin{equation*}
        \begin{aligned}
            e_i\cdot X^{[A]}=&v^{\sum\limits_{p>0}(a_{p,i+1}-a_{p,i})}[a_{0,i+1}]X^{[A+E_{0,i}-E_{0,i+1}]}
            \\&
            +\sum\limits_{\substack{1\leq j\leq m\\a_{j,i+1}>0}} v^{\sum\limits_{p>j}(a_{p,i+1}-a_{p,i})}[a_{j,i+1}]X^{[A+E_{j,i}-E_{j,i+1}]}\\
            &+\sum\limits_{\substack{1\leq j\leq m\\a_{j,-i-1}>0}} v^{\sum\limits_{0\leq p<j}(a_{p,-i-1}-a_{p,-i})+\sum\limits_{0\leq p\leq m}(a_{p,i+1}-a_{p,i})-\delta_{i,0}}[a_{j,-i-1}]X^{[A+E_{j,-i}-E_{j,-i-1}]},\\
            f_i\cdot X^{[A]}=&[2^{\delta_{i,0}}a_{0,i}]v^{\sum\limits_{p>0}(a_{p,-i}-a_{p,-i-1})}X^{[A+E_{0,i+1}-E_{0,i}]}\\
            &+\sum\limits_{\substack{1\leq j\leq m\\a_{j,-i}>0}}v^{\sum\limits_{p>j}(a_{p,-i}-a_{p,-i-1})}[a_{j,-i}]X^{[A+E_{j,-i-1}-E_{j,-i}]}\\
            &+\sum\limits_{\substack{1\leq j\leq m\\a_{j,i}>0}}^mv^{\sum\limits_{0\leq p<j}(a_{p,i}-a_{p,i+1})+\sum\limits_{0\leq p\leq m}(a_{p,-i}-a_{p,-i-1})}[a_{j,i}]X^{[A+E_{j,i+1}-E_{j,i}]},\\
            d_0^{\pm1}X^{[A]}=&v^{\pm\sum\limits_{0\leq p\leq m} a_{p,0}}X^{[A]},
            \quad
            d_h^{\pm1}X^{[A]}=v^{\pm(\sum\limits_{0< p\leq m} (a_{p,h}+a_{p,-h})+a_{0,h})}X^{[A]},\ (0<h\leq n).
        \end{aligned}
    \end{equation*}
    Hence, $\mathbf{U}^{\jmath}_n\cdot \mathrm{Im}\phi\subseteq \mathrm{Im}\phi$.
    \end{proof}
    
    Thus, for $\mathfrak{b}\in \{\imath,\jmath\}$, $\mathscr{S}^{\mathfrak{b}\jmath}_{m|n}$ is actually a $\mathbf{U}^{\jmath}_n$-module.
    We define
    \begin{equation*}
        \begin{aligned}
            E_{i,j}^{\theta}:=E_{i,j}+E_{-i,-j},\quad\mathrm{for}
            \ -m\leq i\leq m,-n\leq j\leq n.
        \end{aligned}
    \end{equation*}  
    \begin{proposition}\label{prop:the left action of Uj on Snij}
        For any $X^{[A]}\in \mathscr{S}^{\mathfrak{b}\jmath}_{m|n}$, $0\leq i< n,0<h\leq n$, we have
        \begin{equation*}
            \begin{aligned}
                e_i\cdot X^{[A]}=&\sum_{\substack{-m\leq j\leq m\\a_{j,{i+1}}>0}}
                v^{\sum\limits_{p>j}(a_{p,i+1}-a_{p,i})}[a_{j,i+1}]X^{[A+E_{j,i}^{\theta}-E_{j,i+1}^{\theta}]},\\
                f_i \cdot X^{[A]}=&\sum_{\substack{-m\leq j\leq 0\\a_{j,{i}}^{\#}>0}}
                v^{\sum\limits_{p<j}(a_{p,i}-a_{p,i+1})}[a_{j,i}-\delta_{i,0}\delta_{j,0}]X^{[A+E_{j,i+1}^{\theta}-E_{j,i}^{\theta}]}\\
                &+\sum_{\substack{0< j\leq m\\a_{j,i}>0}}v^{\sum\limits_{p<j}(a_{p,i}-a_{p,i+1})-\delta_{i,0}}[a_{j,i}]X^{[A+E_{j,i+1}^{\theta}-E_{j,i}^{\theta}]},\\
                d_0^{\pm1}\cdot X^{[A]}=&v^{\pm( \sum\limits_{1\leq p\leq m}a_{p,0}+\frac{a_{0,0}-1}{2})}X^{[A]},~
                d_h^{\pm1}\cdot X^{[A]}=v^{\pm\sum\limits_{-m\leq p\leq m}a_{p,h}}X^{[A]}.
            \end{aligned}
        \end{equation*}
    \end{proposition}
    
    \begin{proof}
    We only give the detailed computation for $f_i\cdot X^{[A]}$. It follows from the proof of Lemma \ref{lemma: Imphi is stable under the action of Uj} and the relations $2^{\delta_{i,0}}a_{0,i}^{\#}=a_{0,i}-\delta_{i,0}$, $\delta_{i,0}a_{0,-i}^{\#}+a_{0,i}^{\#}=a_{0,i}-\delta_{i,0}$, $a_{i,j}=a_{-i,-j}$ that 
        \begin{equation*}
            \begin{aligned}
                &f_i\cdot \phi(X^{[A]})\\
                &=[2^{\delta_{i,0}}a_{0,i}^{\#}]v^{\sum\limits_{p>0}(a_{p,-i}-a_{p,-i-1})}X^{[\rho(A)+E_{0,i+1}-E_{0,i}]}
                \\&+
                \sum_{\substack{1\leq j\leq m\\a_{j,{-i}}>0}}
                v^{\sum\limits_{p>j}(a_{p,-i}-a_{p,-i-1})}[a_{j,-i}]X^{[\rho(A)+E_{j,-i-1}-E_{j,-i}]}\\
                &+\sum_{\substack{1\leq j\leq m\\a_{j,i}>0}}v^{\sum\limits_{0< p<j}(a_{p,i}-a_{p,i+1})+\sum\limits_{0< p\leq m}(a_{p,-i}-a_{p,-i-1})+\delta_{i,0}a_{0,-i}^{\#}+a_{0,i}^{\#}-a_{0,i+1}}[a_{j,i}]X^{[\rho(A)+E_{j,i+1}-E_{j,i}]}\\
                &=[a_{0,i}-\delta_{i,0}]v^{\sum\limits_{p<0}(a_{p,i}-a_{p,i+1})}X^{[\rho(A)+E_{0,i+1}-E_{0,i}]}
                \\
                &+\sum_{\substack{-m\leq j\leq -1\\a_{j,i}>0}}
                v^{\sum\limits_{p<j}(a_{p,i}-a_{p,i+1})}[a_{j,i}]X^{[\rho(A)+E_{-j,-i-1}-E_{-j,-i}]}\\
                &+\sum_{\substack{0<j\leq m\\a_{j,i}>0}}v^{\sum\limits_{p<j}(a_{p,i}-a_{p,i+1})-\delta_{i,0}}[a_{j,i}]X^{[\rho(A)+E_{j,i+1}-E_{j,i}]}.\\
            \end{aligned}
        \end{equation*}
    Note that $E_{i,j}^{\theta}=E_{-i,-j}^{\theta}$, we have
    \begin{equation*}
        \begin{aligned}
            f_i\cdot X^{[A]}=&\sum_{\substack{-m\leq j\leq 0\\a_{j,{i}}^{\#}>0}}
                v^{\sum\limits_{p<j}(a_{p,i}-a_{p,i+1})}[a_{j,i}-\delta_{i,0}\delta_{j,0}]X^{[A+E_{j,i+1}^{\theta}-E_{j,i}^{\theta}]}\\
                &+\sum_{\substack{0< j\leq m\\a_{j,i}>0}}v^{\sum\limits_{p<j}(a_{p,i}-a_{p,i+1})-\delta_{i,0}}[a_{j,i}]X^{[A+E_{j,i+1}^{\theta}-E_{j,i}^{\theta}]}.  
        \end{aligned}
    \end{equation*}
    The computation for the other formulas is similar.
\end{proof}

\subsection{The left action of $\mathbf{U}^{\imath}_n$ on $\mathscr{S}^{\mathfrak{b}\imath}_{m|n}$}\label{subsection: left Ui action on Sbi}
We will construct a $\mathbf{U}^{\imath}_n$-module structure on $\mathscr{S}^{\mathfrak{b}\imath}_{m|n}$ in a similar way .
Let $V'$ be a $N-1$ dimensional $\mathbb{A}$-vector space. We denote by $S(V')\cong \mathbb{A}[X_{-n+\frac{1}{2}},X_{-n+\frac{3}{2}},..., X_{n-\frac{1}{2}}]$ the symmetric algebra of $V'$. For $(\mathbf{a})=(a_{-n+\frac{1}{2}},a_{-n+\frac{3}{2}},...,a_{n-\frac{1}{2}})\in \mathbb{N}^{N-1}$, we set
\begin{equation*}
    \begin{aligned}
        X^{(\mathbf{a})}=X_{-n+\frac{1}{2}}^{a_{{-n+\frac{1}{2}}}}\cdots X_{n-\frac{1}{2}}^{a_{{n-\frac{1}{2}}}}.
    \end{aligned}
\end{equation*}
We equip $S(V')$ with the $\mathbf{U}_v(\mathfrak{gl}_{N-1})$-module structure realized by quantum differential operators in \cite[Example 2.2(1)]{DZ20}. Explicitly, the generators act on the monomial basis as follows:
\begin{equation}\label{equation: action formula on SN-1V}
    \begin{aligned}
        E_i\cdot X^{(\mathbf{a})}=[a_{i+\frac{1}{2}}]X^{(\mathbf{a}+e_{i-\frac{1}{2}}-e_{i+\frac{1}{2}})},\quad
        F_i\cdot X^{(\mathbf{a})}=[a_{i-\frac{1}{2}}]X^{(\mathbf{a}-e_{i-\frac{1}{2}}+e_{i+\frac{1}{2}})},\quad
        D_h\cdot X^{(\mathbf{a})}=v^{a_h}X^{(\mathbf{a})},
    \end{aligned}
\end{equation}
where $i\in \{-n+1,\cdots,n-1\}, h\in \{-n+\frac{1}{2},\cdots,n-\frac{1}{2}\}$ and $[n]=\frac{v^{n}-v^{-n}}{v-v^{-1}}$ is the quantum integer.

Since iquantum group $\mathbf{U}^{\imath}_n$ is a subalgebra of $\mathbf{U}_v(\mathfrak{gl}_{N-1})$,
we can view $S(V')$ as a $\mathbf{U}^{\imath}_n$-module.
Next, we rewrite the elements of $S(V')$ and introduce some constants. For $\mathbf{k}=(k_1,\cdots,k_n)$, $\mathbf{b}=(b_1,\cdots ,b_n)\in \mathbb{N}^n$, we set
\begin{equation*}
    \begin{aligned}
        X_{\mathbf{k},\mathbf{b}}&=X_{-n+\frac{1}{2}}^{k_{n}}\cdots X_{-\frac{1}{2}}^{k_1}X_{\frac{1}{2}}^{b_1}\cdots X_{n-\frac{1}{2}}^{b_n},\quad\\
        K_{\mathbf{k}}=k_1+\cdots&+k_n,\quad
        Q_{\mathbf{k},\mathbf{b}}=\frac{K_{\mathbf{k}}(K_{\mathbf{k}}+1)}{2}+\sum\limits_{1\leq r<s\leq n} b_rk_s,\\
        \quad &C_{\mathbf{k},\mathbf{b}}=v^{-Q_{\mathbf{k},\mathbf{b}}}.
    \end{aligned}
\end{equation*}
Let $\mathcal{M}^{\imath}$ be the subspace of $S(V')$ defined by
\begin{equation*}
    \begin{aligned}
        \mathcal{M}^{\imath}=\mathrm{span}_{\mathbb{A}}\{ R_{\mathbf{k},\mathbf{b}}=X_{\mathbf{k},\mathbf{b}}-C_{\mathbf{k},\mathbf{b}}X_{\mathbf{0},\mathbf{k}+\mathbf{b}}\in S(V') \mid \mathbf{k},\mathbf{b}\in \mathbb{N} ^n\}.
    \end{aligned}
\end{equation*}

\begin{lemma}
 The constants $C_{\mathbf{k},\mathbf{b}}\in \mathbb{A}$ satisfy the following relations:
    \begin{equation}\label{equation: relations of constants Ckb}
        \begin{aligned}
            C_{\mathbf{k},\mathbf{b}+\epsilon_i-\epsilon_{i+1}}&=v^{-k_{i+1}} C_{\mathbf{k},\mathbf{b}},\quad
            C_{\mathbf{k},\mathbf{b}-\epsilon_i+\epsilon_{i+1}}=v^{k_{i+1}} C_{\mathbf{k},\mathbf{b}},\\
            C_{\mathbf{k}+\epsilon_i-\epsilon_{i+1},\mathbf{b}}&=v^{b_i}C_{\mathbf{k},\mathbf{b}},\quad
            C_{\mathbf{k}-\epsilon_i+\epsilon_{i+1},\mathbf{b}}=v^{-b_i}C_{\mathbf{k},\mathbf{b}},\\
            C_{\mathbf{k}+\epsilon_1,\mathbf{b}-\epsilon_1}&=v^{-k_1-1}C_{\mathbf{k},\mathbf{b}},\quad
            C_{\mathbf{k}-\epsilon_1,\mathbf{b}+\epsilon_1}=v^{k_1}C_{\mathbf{k},\mathbf{b}}.
        \end{aligned}
    \end{equation}
\end{lemma}
\begin{proof}
    It follows from straightforward computations.
\end{proof}

\begin{proposition}\label{prop: M is a Ui module}
      The subspace $\mathcal{M}^{\imath}$ is a $\mathbf{U}^{\imath}_n$-submodule of $S(V')$.
\end{proposition}
\begin{proof}
    It suffice to show that  $\mathcal{M}^{\imath}$ is stable under the action of the generators of $\mathbf{U}^{\imath}_n$. The computation is similar to the proof of Proposition \ref{prop: Mj is a Uj module}. It follows from (\ref{equation: action formula on SN-1V}), (\ref{equation: generators of type i}), (\ref{equation: relations of constants Ckb}) that 
    \begin{equation*}
        \begin{aligned}
            e_i\cdot X_{\mathbf{k},\mathbf{b}}
            =&(v^{-k_{i+1}}[b_{i+1}]+v^{b_{i+1}}[k_{i+1}])C_{\mathbf{k},\mathbf{b}}X_{\mathbf{0},\mathbf{k}+\mathbf{b}+\epsilon_i-\epsilon_{i+1}}\\
            &+[b_{i+1}]R_{\mathbf{k},\mathbf{b}+\epsilon_i-\epsilon_{i+1}}+v^{b_{i+1}-b_i}[k_{i+1}]R_{\mathbf{k}+\epsilon_i-\epsilon_{i+1},\mathbf{b}}.
       \\
            e_i.(C_{\mathbf{k},\mathbf{b}}X_{\mathbf{0},\mathbf{k}+\mathbf{b}})
            &=[k_{i+1}+b_{i+1}]C_{\mathbf{k},\mathbf{b}}X_{\mathbf{0},\mathbf{k}+\mathbf{b}+\epsilon_i-\epsilon_{i+1}}\\
            &=(v^{-k_{i+1}}[b_{i+1}]+v^{b_{i+1}}[k_{i+1}])C_{\mathbf{k},\mathbf{b}}X_{\mathbf{0},\mathbf{k}+\mathbf{b}+\epsilon_i-\epsilon_{i+1}}.
        \end{aligned}
    \end{equation*}
    Thus,
    \begin{equation*}
        e_i.R_{\mathbf{k},\mathbf{b}}=[b_{i+1}]R_{\mathbf{k},\mathbf{b}+\epsilon_i-\epsilon_{i+1}}+v^{b_{i+1}-b_i}[k_{i+1}]R_{\mathbf{k}+\epsilon_i-\epsilon_{i+1},\mathbf{b}}\in \mathcal{M}^{\imath}.
    \end{equation*}
    Similarly, we have
    \begin{equation*}
        \begin{aligned}
            &f_i\cdot X_{\mathbf{k},\mathbf{b}}
            =(v^{-b_i}[k_i]+v^{k_i}[b_i])C_{\mathbf{k},\mathbf{b}}X_{\mathbf{0},\mathbf{k}+\mathbf{b}-\epsilon_i+\epsilon_{i+1}}+[k_i]R_{\mathbf{k}-\epsilon_i+\epsilon_{i+1},\mathbf{b}}+v^{k_i-k_{i+1}}[b_i]R_{\mathbf{k},\mathbf{b}-\epsilon_i+\epsilon_{i+1}}.
            \\
            &f_i.(C_{\mathbf{k},\mathbf{b}}X_{\mathbf{0},\mathbf{k}+\mathbf{b}})
            =[k_i+b_i]C_{\mathbf{k},\mathbf{b}}X_{\mathbf{0},\mathbf{k}+\mathbf{b}-\epsilon_i+\epsilon_{i+1}}=(v^{-b_i}[k_i]+v^{k_i}[b_i])C_{\mathbf{k},\mathbf{b}}X_{\mathbf{0},\mathbf{k}+\mathbf{b}-\epsilon_i+\epsilon_{i+1}}.
        \end{aligned}
    \end{equation*}
    Thus,
    \begin{equation*}
        \begin{aligned}
            f_i.R_{\mathbf{k},\mathbf{b}}=[k_i]R_{\mathbf{k}-\epsilon_i+\epsilon_{i+1},\mathbf{b}}+v^{k_i-k_{i+1}}[b_i]R_{\mathbf{k},\mathbf{b}-\epsilon_i+\epsilon_{i+1}}\in \mathcal{M}^{\imath}.
        \end{aligned}
    \end{equation*}
   Next, we compute for $t_0\cdot X_{\mathbf{k},\mathbf{b}}$ as follows:
   \begin{equation*}
       \begin{aligned}
           t_0\cdot X_{\mathbf{k},\mathbf{b}}
           =&[b_1]X_{\mathbf{k}+\epsilon_1,\mathbf{b}-\epsilon_1}+v^{b_1-k_1+1}[k_1]X_{\mathbf{k}-\epsilon_1,\mathbf{b}+\epsilon_1}+v^{b_1-k_1}X_{\mathbf{k},\mathbf{b}}\\
           =&(v^{-k_1-1}[b_1]+v^{b_1+1}[k_1]+v^{b_1-k_1})C_{\mathbf{k},\mathbf{b}}X_{\mathbf{0},\mathbf{k}+\mathbf{b}}\\
           &+[b_1]R_{\mathbf{k}+\epsilon_1,\mathbf{b}-\epsilon_1}+v^{b_1-k_1+1}[k_1]R_{\mathbf{k}-\epsilon_1,\mathbf{b}+\epsilon_1}+v^{b_1-k_1}R_{\mathbf{k},\mathbf{b}}.
       \end{aligned}
   \end{equation*}
    \begin{equation*}
        \begin{aligned}
            t_0.(C_{\mathbf{k},\mathbf{b}}X_{\mathbf{0},\mathbf{k}+\mathbf{b}})=&[k_1+b_1]C_{\mathbf{k},\mathbf{b}}X_{\epsilon_1,\mathbf{k}+\mathbf{b}-\epsilon_1}+v^{k_1+b_1}C_{\mathbf{k},\mathbf{b}}X_{\mathbf{0},\mathbf{k}+\mathbf{b}}\\
            =&[k_1+b_1+1]C_{\mathbf{k},\mathbf{b}}X_{\mathbf{0},\mathbf{k}+\mathbf{b}}+[k_1+b_1]C_{\mathbf{k},\mathbf{b}}R_{\epsilon_1,\mathbf{k}+\mathbf{b}-\epsilon_1}.
        \end{aligned}
    \end{equation*}
    Note that $[k_1+b_1+1]=v^{-k_1-1}[b_1]+v^{b_1}[k_1+1]=v^{-k_1-1}[b_1]+v^{b_1+1}[k_1]+v^{b_1-k_1}$, so we have
    \begin{equation*}
        \begin{aligned}
            t_0.R_{\mathbf{k},\mathbf{b}}=[b_1]R_{\mathbf{k}+\epsilon_1,\mathbf{b}-\epsilon_1}+v^{b_1-k_1+1}[k_1]R_{\mathbf{k}-\epsilon_1,\mathbf{b}+\epsilon_1}+v^{b_1-k_1}R_{\mathbf{k},\mathbf{b}}-[k_1+b_1]C_{\mathbf{k},\mathbf{b}}R_{\epsilon_1,\mathbf{k}+\mathbf{b}-\epsilon_1}\in \mathcal{M}^{\imath}.
        \end{aligned}
    \end{equation*}
    Finally, we have
    \begin{equation*}
        \begin{aligned}
        d_j^{\pm1}\cdot X_{\mathbf{k},\mathbf{b}}=v^{\pm(b_j+k_j)}X_{\mathbf{k},\mathbf{b}}
        &=v^{\pm(b_j+k_j)}C_{\mathbf{k},\mathbf{b}}X_{\mathbf{0},\mathbf{k}+\mathbf{b}}+v^{\pm(b_j+k_j)}R_{\mathbf{k},\mathbf{b}}
        \\&=d_j^{\pm1}(C_{\mathbf{k},\mathbf{b}}X_{\mathbf{0},\mathbf{k}+\mathbf{b}})+v^{\pm(b_j+k_j)}R_{\mathbf{k},\mathbf{b}}.
        \end{aligned}
    \end{equation*}
    \begin{equation*}
        \begin{aligned}
            d_j^{\pm}.R_{\mathbf{k},\mathbf{b}}=v^{\pm(b_j+k_j)}R_{\mathbf{k},\mathbf{b}}\in \mathcal{M}^{\imath}.
        \end{aligned}
    \end{equation*}
\end{proof}
Let $S^{\imath}(V'):=S(V')/\mathcal{M}^{\imath}$, then $S^{\imath}(V')$ is a $\mathbf{U}^{\imath}_n$ module by the above proposition. By abuse of notation, we also denote by $X^{(\mathbf{a})}$($X_{\mathbf{k},\mathbf{b}}$) the elements of $S^{\imath}(V')$ and always remember the relations $R_{\mathbf{k},\mathbf{b}}=0$. It follows from the definition of $\mathcal M^\imath$ that $\{X_{\mathbf0,\mathbf{a}}, \mathbf{a}\in\mathbb N^n\}$ form an $\mathbb A$-basis of $S^\imath(V)$.
\begin{corollary}
For any $X_{\mathbf{0},\mathbf{a}}\in S^{\imath}(V')$ and $e_i$, $f_i$, $t_0$, $d_j\in \mathbf{U}^{\imath}_n$, we have 
    \begin{equation}\label{equation: the action of type i on X0,a}
        \begin{aligned}
            e_i\cdot X_{\mathbf{0},\mathbf{a}}&=[a_{i+1}]X_{\mathbf{0},\mathbf{a}+\epsilon_i-\epsilon_{i+1}},\quad
            f_i\cdot X_{\mathbf{0},\mathbf{a}}=[a_i]X_{\mathbf{0},\mathbf{a}-\epsilon_i+\epsilon_{i+1}},\\
            t_0\cdot X_{\mathbf{0},\mathbf{a}}&=[a_1+1]X_{\mathbf{0},\mathbf{a}},\quad \quad\quad
            d_j\cdot X_{\mathbf{0},\mathbf{a}}=v^{a_j}X_{\mathbf{0},\mathbf{a}}.
        \end{aligned}
    \end{equation}
\end{corollary}
\begin{proof}
    It follows from the proof of Proposition \ref{prop: M is a Ui module} directly.
\end{proof}
 We set $S^{\imath}_M(V'):=(S(V'))^{\otimes m}\otimes S^{\imath} \otimes  (S(V'))^{\otimes m}$ and write the elements of $S^{\imath}_M(V')$ as
\begin{equation*}
    \begin{aligned}
        X^{[A]}=X^{(r_{-m})}\otimes\cdots\otimes X^{(r_m)},
    \end{aligned}
\end{equation*}
where $A \in \mathrm{Mat}_{M\times (N-1)}{(\mathbb{N})} $ and $r_i$ is the $(i+m+1)$-row vector of $A$.
Next, we define the $\mathbf{U}^{\imath}_n$-module structure on $S^{\imath}_M(V')$. Since $\Delta^m(\mathbf{U}^{\imath}_n)\subseteq\mathbf{U}^{\imath}_n\otimes \Delta^{m-1}(\mathbf{U}_v(\mathfrak{gl}_{N-1}))$ and $S^{\imath}_M(V')$(respectively, $S(V')$) is $\mathbf{U}^{\imath}_n$(respectively, $\mathbf{U}_v(\mathfrak{gl}_{N-1}))$)module, we can define the module structure via comultiplication. Explicitly, for $x\in \mathbf{U}^{\imath}_n, X^{[A]}\in S^{\imath}_M(V')$, we define $x\cdot X^{[A]}=(1^{\otimes m}\otimes \Delta^m(x))\cdot X^{[A]} $.

Next, we define the $\mathbf{U}^{\imath}_n$-module structure on $\mathscr{S}^{\mathfrak{b}\imath}_{m|n}$ for $\mathfrak{b}\in \{\imath,\jmath\}$.
For $A\in \Xi^{\mathfrak{b} \imath}_{m|n}$, we associate it with a matrix $\rho(A)\in \mathrm{Mat}_{M\times (N-1)}(\mathbb{N})$ by deleting the $0$-th column of $A$ and then retaining only half of the remaining entries. Moreover, we shift the column indices of the matrices in $\mathrm{Mat}_{M\times (N-1)}(\mathbb{N})$ to half-integers in order to preserve the symmetry. More precisely, the map $\rho: \Xi^{\mathfrak{b}\imath}_{m|n}\xrightarrow[]{}\mathrm{Mat}_{M\times (N-1)}(\mathbb{N})$ satisfies
\begin{equation*}
    \begin{aligned}
        \rho(A)_{i,j-\frac{1}{2}}=a_{i,j},\quad& for \ i\geq 0,j>0,\quad
        \rho(A)_{i,-j+\frac{1}{2}}=a_{i,-j}\quad for \ i> 0,j>0,\\
        &\mathrm{other} \  \mathrm{elements} \ \mathrm{are} \ 0.
    \end{aligned}
\end{equation*}
Then, we induce an injection
$\phi :\mathscr{S}^{\mathfrak{b}\imath}_{m|n}\xrightarrow{} S^{\imath}_M(V'),
X^{[A]}\mapsto X^{[\rho(A)]}$. For any $x\in \mathbf{U}^{\imath}_n$, $X^{[A]}\in \mathscr{S}^{\mathfrak{b}\imath}_{m|n}$, we define 
\begin{equation*}
    \begin{aligned}
        x\cdot X^{[A]}:=\phi^{-1}( x\cdot\phi(X^{[A]})). 
    \end{aligned}
\end{equation*}
Also, we need to check that $\mathrm{Im}\phi$ is stable under the action of $\mathbf{U}^{\imath}_n$.

\begin{lemma}\label{lemma: Imphi is stable under the action of Ui}
Under the above notation,
    $\mathbf{U}^{\imath}_n\cdot \mathrm{Im}\phi\subseteq\mathrm{Im}\phi.$
\end{lemma}
\begin{proof}
    By induction, we have
    \begin{equation*}
        \begin{aligned}
            \Delta^{m}(e_i)&=e_i\otimes (K_i^{-1})^{\otimes m}+\sum\limits_{j=1}^m 1^{\otimes j}\otimes E_i \otimes(K_i^{-1})^{\otimes (m-j)}+\sum\limits_{j=1}^m(k_{i}^{-1})^{\otimes j}\otimes K_i^{-1}F_{-i}\otimes (K_i^{-1})^{\otimes (m-j)},\\
            \Delta^{m}(f_i)&=f_i\otimes (K_{-i}^{-1})^{\otimes m}+\sum\limits_{j=1}^m 1^{\otimes j}\otimes E_{-i} \otimes(K_{-i}^{-1})^{\otimes (m-j)}+\sum\limits_{j=1}^m(k_{i})^{\otimes j}\otimes F_{i}K_{-i}^{-1}\otimes (K_{-i}^{-1})^{\otimes (m-j)},\\
            \Delta^m(t_0)&=t_0\otimes(K_0^{-1})^{\otimes m}+\sum\limits_{j=1}^m1^{\otimes j }\otimes (E_0+vF_0K_0^{-1})\otimes(K_0^{-1})^{\otimes (m-j)}.
        \end{aligned}
    \end{equation*}
    Note that $\mathrm{Im}\phi=\mathrm{span}\{X^{[A]}\in S^{\imath}_M(V')\mid A\in \rho(\Xi^{\mathfrak{b} \imath}_{m|n})\}$. For any $X^{[A]}\in \mathrm{Im}\phi$, it follows from (\ref{equation: action formula on SN-1V}), (\ref{equation: the action of type i on X0,a}) that 
    \begin{equation*}
        \begin{aligned}
            e_i&\cdot X^{[A]}\\
            =&v^{\sum\limits_{p>0}(a_{p,i+\frac{1}{2}}-a_{p,i-\frac{1}{2}})}[a_{0,i+\frac{1}{2}}]X^{[A+E_{0,i-\frac{1}{2}}-E_{0,i+\frac{1}{2}}]}
            \\
            &+\sum\limits_{\substack{1\leq j\leq m\\a_{j,i+\frac{1}{2}}>0}}v^{\sum\limits_{p>j}(a_{p,i+\frac{1}{2}}-a_{p,i-\frac{1}{2}})}[a_{j,i+\frac{1}{2}}]X^{[A+E_{j,i-\frac{1}{2}}-E_{j,i+\frac{1}{2}}
            ]}\\
            &+\sum\limits_{\substack{1\leq j\leq m\\a_{j,-i-\frac{1}{2}}>0}}v^{\sum\limits_{0<p<j}(a_{p,-i-\frac{1}{2}}-a_{p,-i+\frac{1}{2}})+\sum\limits_{0\leq p\leq m}(a_{p,i+\frac{1}{2}}-a_{p,i-\frac{1}{2}})}[a_{j,-i-\frac{1}{2}}]X^{[A-E_{j,-i-\frac{1}{2}}+E_{j,-i+\frac{1}{2}}]}\\
            =&\sum\limits_{\substack{0\leq j\leq m\\a_{j,i+\frac{1}{2}}>0}}v^{\sum\limits_{p>j}(a_{p,i+\frac{1}{2}}-a_{p,i-\frac{1}{2}})}[a_{j,i+\frac{1}{2}}]X^{[A+E_{j,i-\frac{1}{2}}-E_{j,i+\frac{1}{2}}
            ]}\\
            &+\sum\limits_{\substack{1\leq j\leq m\\a_{j,-i-\frac{1}{2}}>0}}v^{\sum\limits_{0<p<j}(a_{p,-i-\frac{1}{2}}-a_{p,-i+\frac{1}{2}})+\sum\limits_{0\leq p\leq m}(a_{p,i+\frac{1}{2}}-a_{p,i-\frac{1}{2}})}[a_{j,-i-\frac{1}{2}}]X^{[A-E_{j,-i-\frac{1}{2}}+E_{j,-i+\frac{1}{2}}]},\\
    \end{aligned}
        \end{equation*}
    Similarly, we have        
    \begin{equation*}
        \begin{aligned}
            f_i \cdot X^{[A]}
            =&v^{\sum\limits_{p>0}(a_{p,-i+\frac{1}{2}}-a_{p,-i-\frac{1}{2}})}[a_{0,i-\frac{1}{2}}]X^{[A-E_{0,i-\frac{1}{2}}+E_{0,i+\frac{1}{2}}]}
            \\&+\sum\limits_{\substack{1\leq j\leq m\\ a_{j,-i+\frac{1}{2}}>0}}v^{\sum\limits_{p>j}(a_{p,-i+\frac{1}{2}}-a_{p,-i-\frac{1}{2}})}[a_{j,-i+\frac{1}{2}}]X^{[A+E_{j,-i-\frac{1}{2}}-E_{j,-i+\frac{1}{2}}]}\\
            &+\sum\limits_{\substack{1\leq j\leq m\\ a_{j,i-\frac{1}{2}}>0}}v^{\sum\limits_{0<p<j}(a_{p,i-\frac{1}{2}}-a_{p,i+\frac{1}{2}})+\sum\limits_{0 \leq p \leq m}(a_{p,-i+\frac{1}{2}}-a_{p,-i-\frac{1}{2}})}[a_{j,i-\frac{1}{2}}]X^{[A-E_{j,i-\frac{1}{2}}+E_{j,i+\frac{1}{2}}]}.
            \\
            t_0 \cdot X^{[A]}=&v^{\sum\limits_{p>0}(a_{p,\frac{1}{2}}-a_{p,-\frac{1}{2}})}[a_{0,\frac{1}{2}}+1]X^{[A]}+\sum\limits_{\substack{1\leq j \leq m \\ a_{j,\frac{1}{2}}>0}} v^{\sum\limits_{p>j}(a_{p,\frac{1}{2}}-a_{p,-\frac{1}{2}})}[a_{j,\frac{1}{2}}]X^{[A+E_{j,-{\frac{1}{2}}}-E_{j,\frac{1}{2}}]}\\
            &\quad \quad \quad \quad \quad \quad \quad \quad \quad \quad \quad \quad \   +\sum\limits_{\substack{1\leq j \leq m\\ a_{j,-\frac{1}{2}}>0}} v^{\sum\limits_{p\geq j}(a_{p,\frac{1}{2}}-a_{p,-\frac{1}{2}})+1}[a_{j,-\frac{1}{2}}]X^{[A-E_{j,-{\frac{1}{2}}}+E_{j,\frac{1}{2}}]}.\\
            d_j^{\pm1} \cdot X^{[A]}&=v^{\pm(\sum\limits_{p>0}(a_{p,j-\frac{1}{2}}+a_{p,-j+\frac{1}{2}})+a_{0,j-\frac{1}{2}})}X^{[A]}.
        \end{aligned}
    \end{equation*}
\end{proof}
\begin{proposition}\label{prop: the left action of Ui on Snji}
    For any $X^{[A]}\in \mathscr{S}^{\mathfrak{b}\imath}_{m|n}$, $e_i$, $f_i$, $t_0$, $d_j\in  \mathbf{U}^{\jmath}_n$, we have 
    \begin{equation*}
        \begin{aligned}
            e_i\cdot X^{[A]}=&\sum\limits_{\substack{-m\leq j\leq m\\a_{j,i+1}>0}}v^{\sum\limits_{p>j}(a_{p,i+1}-a_{p,i})}[a_{j,i+1}]X^{[A+E_{j,i}^{\theta}-E_{j,i+1}
            ^{\theta}]},\\
            f_i\cdot X^{[A]}=&\sum\limits_{\substack{-m\leq j\leq m\\ a_{j,i}>0}}v^{\sum\limits_{p<j}(a_{p,i}-a_{p,i+1})}[a_{j,i}]X^{[A+E_{j,i+1}^{\theta}-E_{j,i}^{\theta}]},\\
            t_0\cdot X^{[A]}=&v^{\sum\limits_{p>0}(a_{p,1}-a_{p,-1})+a_{0,1}}X^{[A]}
            \\
            &+\sum\limits_{\substack{-m\leq j \leq m \\ a_{j,1}>0}} v^{\sum\limits_{p>j}(a_{p,1}-a_{p,-1})-a_{j,0}+\sum\limits_{k<0}\delta_{j,k}}[a_{j,1}]X^{[A+E_{j,-1}^{\theta}-E_{j,1}^{\theta}]},\\
            d_j^{\pm1}\cdot X^{[A]}=&v^{\pm \sum\limits_{-m\leq p\leq m}a_{p,j}}X^{[A]}.
        \end{aligned}
    \end{equation*}
\end{proposition}
\begin{proof}
    By the relation $a_{i,j}=a_{-i,-j}$ and the proof of Lemma \ref{lemma: Imphi is stable under the action of Ui}, we have    
    \begin{equation*}
        \begin{aligned}
            &t_0 \cdot\phi( X^{[A]})
            \\=&v^{\sum\limits_{p>0}(a_{p,1}-a_{p,-1})}[a_{0,1}+1]X^{[\rho(A)]}+\sum\limits_{\substack{1\leq j \leq m \\ a_{j,1}>0}} v^{\sum\limits_{p>j}(a_{p,1}-a_{p,-1})}[a_{j,1}]X^{[\rho(A)+E_{j,-{\frac{1}{2}}}-E_{j,\frac{1}{2}}]}\\
            &+\sum\limits_{\substack{1\leq j \leq m \\ a_{j,-1}>0}} v^{\sum\limits_{1\leq p \leq m}(a_{p,1}-a_{p,-1})+\sum\limits_{p<j}(a_{p,-1}-a_{p,1})+1}[a_{j,-1}]X^{[\rho(A)-E_{j,-{\frac{1}{2}}}+E_{j,\frac{1}{2}}]}\\
            =&v^{\sum\limits_{p>0}(a_{p,1}-a_{p,-1})}(v^{a_{0,1}}+v^{-1}[a_{0,1}])X^{[\rho(A)]}+\sum\limits_{\substack{1\leq j \leq m \\ a_{j,1}>0}} v^{\sum\limits_{p>j}(a_{p,1}-a_{p,-1})}[a_{j,1}]X^{[\rho(A)+E_{j,-{\frac{1}{2}}}-E_{j,\frac{1}{2}}]}\\
            &+\sum\limits_{\substack{-m\leq j \leq -1 \\ a_{j,1}>0}} v^{\sum\limits_{p>j}(a_{p,1}-a_{p,-1})+1}[a_{j,1}]X^{[\rho(A)-E_{-j,-{\frac{1}{2}}}+E_{-j,\frac{1}{2}}]}.
        \end{aligned}
    \end{equation*}
    Note that $a_{0,0}=1$ and $a_{j,0}=0(j\neq 0)$, we have
    \begin{equation*}
        \begin{aligned}
             t_0\cdot X^{[A]}=&\phi^{-1}(t_0\cdot \phi(X^{[A]}))\\
             =&v^{\sum\limits_{p>0}(a_{p,1}-a_{p,-1})}(v^{a_{0,1}}+v^{-1}[a_{0,1}])X^{[A]}+\sum\limits_{\substack{1\leq j \leq m \\ a_{j,1}>0}} v^{\sum\limits_{p>j}(a_{p,1}-a_{p,-1})}[a_{j,1}]X^{[A+E_{j,-1}^{\theta}-E_{j,1}^{\theta}]}\\
            &+\sum\limits_{\substack{-m\leq j \leq -1 \\ a_{j,1}>0}} v^{\sum\limits_{p>j}(a_{p,1}-a_{p,-1})+1}[a_{j,1}]X^{[A-E_{-j,-1}^{\theta}+E_{-j,1}^{\theta}]}\\
            =&v^{\sum\limits_{p>0}(a_{p,1}-a_{p,-1})+a_{0,1}}X^{[A]}+\sum\limits_{\substack{-m\leq j \leq m \\ a_{j,1}>0}} v^{\sum\limits_{p>j}(a_{p,1}-a_{p,-1})-a_{j,0}+\sum\limits_{k<0}\delta_{j,k}}[a_{j,1}]X^{[A+E_{j,-1}^{\theta}-E_{j,1}^{\theta}]}.
        \end{aligned}
        \end{equation*}
        The computation for the other formulas is similar to the proof of \ref{prop:the left action of Uj on Snij}.
\end{proof}

\subsection{The right action of $\mathbf{U}^{\mathfrak{b}}_m$ on  $\mathscr{S}^{\mathfrak{b}\mathfrak{c}}_{m|n}$}\label{subsection: right Ub action Sbc}
By the definition of $\mathscr{S}^{\mathfrak{b}\mathfrak{c}}_{m|n}$, there is an $\mathbb{A}$-linear isomorphism:
\begin{equation*}
    \begin{aligned}
        \sigma:\mathscr{S}^{\mathfrak{b}\mathfrak{c}}_{m|n}&\xrightarrow[]{}\mathscr{S}^{\mathfrak{c}\mathfrak{b}}_{n|m},\quad
        X^{[A]}\mapsto X^{[A^T]},
    \end{aligned}
\end{equation*}
where $A^T$ is the transpose of $A$.
We restrict the algebra anti-automorphism $\omega$ of $\mathbf{U}_v(\mathfrak{gl}_{M-1})$ given in  (\ref{equation: an anti-automorphsim of quantum group}) to the subalgebra $\mathbf{U}^{\imath}_m$, then
\begin{equation*}
    \begin{aligned}
        \omega(e_i)=vf_ik_i^{-1},\quad\omega(f_i)=ve_ik_i,\quad \omega(t_0)=t_0,\quad \omega(d_j^{\pm})=d_j^{\pm},
    \end{aligned}
\end{equation*}
where $0<i<m,0<j\leq m$. 
Hence, $\omega(\mathbf{U}^{\imath}_m)=\mathbf{U}^{\imath}_m$ and $\omega$ is  an algebra anti-automorphism of $\mathbf{U}^{\imath}_m$.
Similarly, we restrict the algebra anti-automorphism $\omega$ of $\mathbf{U}_v(\mathfrak{gl}_M)$ given in  (\ref{equation: an anti-automorphsim of quantum group}) to the subalgebra $\mathbf{U}^{\jmath}_m$, then 
\begin{equation*}
    \begin{aligned}
        \omega(e_i)=vf_ik_i^{-1},\quad\omega(f_i)=v^{1+\delta_{i,0}}e_ik_i,\quad \omega(d_j^{\pm1})=d_j^{\pm1}.
    \end{aligned}
\end{equation*}
where $0\leq i<m,0\leq j\leq m$. Hence, $\omega(\mathbf{U}^{\jmath}_m)=\mathbf{U}^{\jmath}_m$ and $\omega$  is an algebra anti-automorphism of $\mathbf{U}^{\jmath}_m$.
Since $\mathscr{S}^{\mathfrak{c}\mathfrak{b}}_{n|m}$ is a $\mathbf{U}^{\mathfrak{b}}_m$ module, we define the right $\mathbf{U}^{\mathfrak{b}}_m$-module structure on $\mathscr{S}^{\mathfrak{b}\mathfrak{c}}_{m|n}$ as follow:
\begin{equation*}
    \begin{aligned}
        X^{[A]}\cdot x=\sigma^{-1}(\omega(x)\cdot \sigma(X^{[A]})),\quad \mathrm{for}\ X^{[A]}\in \mathscr{S}^{\mathfrak{b}\mathfrak{c}}_{m|n},x\in \mathbf{U}^{\mathfrak{b}}_m.
    \end{aligned}
\end{equation*}

We now give the explicit formulas for the right actions.
\begin{proposition}\label{prop:the right action of Ub on Snbc}
     Let $X^{[A]}\in \mathscr{S}^{\mathfrak{b}\mathfrak{c}}_{m|n}$.
    For $e_i,f_i,d_h\in \mathbf{U}^{b}_m$, we have
    \begin{equation*}
        \begin{aligned}
            X^{[A]}\cdot e_i=&\sum_{\substack{-n\leq j\leq 0\\a_{i,j}^{\#}>0}}
            v^{\sum\limits_{p\geq j}(a_{i+1,p}-a_{i,p})+1+\delta_{i,0}}[a_{i,j}-\delta_{i,0}\delta_{j,0}]X^{[A+E_{i+1,j}^{\theta}-E_{i,j}^{\theta}]}\\
            &+\sum_{\substack{0< j\leq n\\a_{i,j}>0}}v^{\sum\limits_{p\geq j}(a_{i+1,p}-a_{i,p})+1}[a_{i,j}]X^{[A+E_{i+1,j}^{\theta}-E_{i,j}^{\theta}]},\\
            X^{[A]}\cdot f_i=&\sum_{\substack{-n\leq j\leq n\\a_{i+1,j}>0}}
            v^{\sum\limits_{p\leq j}(a_{i,p}-a_{i+1,p})+1}[a_{i+1,j}]X^{[A+E_{i,j}^{\theta}-E_{i+1,j}^{\theta}]}.
            \\
           X^{[A]}\cdot d_h^{\pm1}=&v^{\pm\sum\limits_{-n\leq p\leq n}a_{p,h}}X^{[A]},\ (0<h\leq m).
        \end{aligned}
    \end{equation*}
If $\mathfrak{b}=\jmath$, then for $d_0\in \mathbf{U}^{\jmath}_m$, we have
    \begin{equation*}
        \begin{aligned}
            X^{[A]}\cdot d_0^{\pm1}=&v^{\pm( \sum\limits_{1\leq p\leq m}a_{p,0}+\frac{a_{0,0}-1}{2})}X^{[A]}.
        \end{aligned}
    \end{equation*}
\end{proposition}
If $\mathfrak{b}=\imath$, then for $t_0\in \mathbf{U}^{\imath}_m$, we have
    \begin{equation*}
        \begin{aligned}
            X^{[A]}\cdot t_0
        &=v^{\sum\limits_{p>0}(a_{1,p}-a_{-1,p})+a_{1,0}}X^{[A]}+\sum\limits_{\substack{-n\leq j \leq n \\ a_{1,j}>0}} v^{\sum\limits_{p>j}(a_{1,p}-a_{-1,p})-a_{0,j}+\sum\limits_{k<0}\delta_{j,k}}[a_{1,j}]X^{[A+E_{-1,j}^{\theta}-E_{1,j}^{\theta}]}.
        \end{aligned}
    \end{equation*}
\begin{proof}
    By Proposition \ref{prop: the left action of Ui on Snji}, for $e_i\in \mathbf{U}^{\imath}_m$ and any $X^{[A]}\in \mathscr{S}^{\imath\mathfrak{c}}_{m|n}$, we have 
    \begin{equation*}
        \begin{aligned}
            \omega(e_i)\cdot \sigma(X^{[A]})
            &=vf_id_i^{-1}d_{i+1}\cdot X^{[A^T]}\\
            &=\sum\limits_{\substack{-n\leq j\leq n\\ a_{i,j}>0}}v^{\sum\limits_{p<j}(a_{i,p}-a_{i+1,p})+\sum\limits_{-n\leq p \leq n}(a_{i+1,p}-a_{i,p})+1}[a_{i,j}]X^{[A^T+E_{j,i+1}^{\theta}-E_{j,i}^{\theta}]}\\
            &=\sum\limits_{\substack{-n\leq j\leq n\\ a_{i,j}>0}}v^{\sum\limits_{p\geq j}(a_{i+1,p}-a_{i,p})+1}[a_{i,j}]X^{[A^T+E_{j,i+1}^{\theta}-E_{j,i}^{\theta}]}.
        \end{aligned}
    \end{equation*}
    Thus,
    \begin{equation*}
        \begin{aligned}
            X^{[A]}\cdot e_i
            &=\sigma^{-1}(\omega(e_i)\cdot \sigma(X^{[A]}))=\sum\limits_{\substack{-n\leq j\leq n\\ a_{i,j}>0}}v^{\sum\limits_{p\geq j}(a_{i+1,p}-a_{i,p})+1}[a_{i,j}]X^{[A+E_{i+1,j}^{\theta}-E_{i,j}^{\theta}]}.
        \end{aligned}
    \end{equation*}

    On the other hand, it follows from Proposition \ref{prop:the left action of Uj on Snij} that for any $X^{[A]}\in \mathscr{S}^{\jmath\mathfrak{c}}_{m|n}$, we have
    \begin{equation*}
        \begin{aligned}
            \omega(e_0)\cdot \sigma(X^{[A]})
            =&vf_0d_0^{-2}d_1\cdot X^{[A^T]}\\
            =&\sum_{\substack{-n\leq j\leq 0\\a_{0,j}^{\#}>0}}
            v^{\sum\limits_{p<j}(a_{0,p}-a_{1,p})+\sum\limits_{-n\leq p \leq n}a_{1,p}-2\sum\limits_{1\leq p \leq n}a_{0,p}-a_{0,0}+1+1}[a_{0,j}-\delta_{j,0}]X^{[A^T+E_{j,1}^{\theta}-E_{j,0}^{\theta}]}\\
            &+\sum_{\substack{0< j\leq n\\a_{0,j}>0}}v^{\sum\limits_{p<j}(a_{0,p}-a_{1,p})+\sum\limits_{-n\leq p \leq n}a_{1,p}-2\sum\limits_{1\leq p \leq n}a_{0,p}-a_{0,0}+1}[a_{0,j}]X^{[A^T+E_{j,1}^{\theta}-E_{j,0}^{\theta}]}\\
            =&\sum_{\substack{-n\leq j\leq 0\\a_{0,j}^{\#}>0}}
            v^{\sum\limits_{p\geq j}(a_{1,p}-a_{0,p})+2}[a_{0,j}-\delta_{j,0}]X^{[A^T+E_{j,1}^{\theta}-E_{j,0}^{\theta}]}\\
            &+\sum_{\substack{0< j\leq n\\a_{0,j}>0}}v^{\sum\limits_{p\geq j}(a_{1,p}-a_{0,p})+1}[a_{0,j}]X^{[A^T+E_{j,1}^{\theta}-E_{j,0}^{\theta}]}.\\
        \end{aligned}
    \end{equation*}
    For $i\neq 0$, we have 
    \begin{equation*}
        \begin{aligned}
            \omega(e_i)\cdot \sigma(X^{[A]})
            =&vf_id_i^{-1}d_{i+1}\cdot X^{[A^T]}\\
            =&\sum_{\substack{-n\leq j\leq 0\\a_{i,j}^{\#}>0}}
            v^{\sum\limits_{p<j}(a_{i,p}-a_{i+1,p})+\sum\limits_{-n\leq p \leq n}(a_{i+1,p}-a_{i,p})+1}[a_{i,j}-\delta_{i,0}\delta_{j,0}]X^{[A^T+E_{j,i+1}^{\theta}-E_{j,i}^{\theta}]}\\
            &+\sum_{\substack{0< j\leq n\\a_{i,j}>0}}v^{\sum\limits_{p<j}(a_{i,p}-a_{i+1,p})+\sum\limits_{-n\leq p \leq n}(a_{i+1,p}-a_{i,p})+1}[a_{i,j}]X^{[A^T+E_{j,i+1}^{\theta}-E_{j,i}^{\theta}]}\\
            =&\sum_{\substack{-n\leq j\leq 0\\a_{i,j}^{\#}>0}}
            v^{\sum\limits_{p\geq j}(a_{i+1,p}-a_{i,p})+1}[a_{i,j}-\delta_{i,0}\delta_{j,0}]X^{[A^T+E_{j,i+1}^{\theta}-E_{j,i}^{\theta}]}\\
            &+\sum_{\substack{0< j\leq n\\a_{i,j}>0}}v^{\sum\limits_{p\geq j}(a_{i+1,p}-a_{i,p})+1}[a_{i,j}]X^{[A^T+E_{j,i+1}^{\theta}-E_{j,i}^{\theta}]}.
        \end{aligned}
    \end{equation*}
    Combining the two cases above, we obtain 
    \begin{equation*}
        \begin{aligned}
            X^{[A]}\cdot e_i=&\sigma^{-1}(\omega(e_i)\cdot \sigma(X^{[A]}))\\
            =&\sum_{\substack{-n\leq j\leq 0\\a_{i,j}^{\#}>0}}
            v^{\sum\limits_{p\geq j}(a_{i+1,p}-a_{i,p})+1+\delta_{i,0}}[a_{i,j}-\delta_{i,0}\delta_{j,0}]X^{[A+E_{i+1,j}^{\theta}-E_{i,j}^{\theta}]}\\
            &+\sum_{\substack{0< j\leq n\\a_{i,j}>0}}v^{\sum\limits_{p\geq j}(a_{i+1,p}-a_{i,p})+1}[a_{i,j}]X^{[A+E_{i+1,j}^{\theta}-E_{i,j}^{\theta}]}.
        \end{aligned}
    \end{equation*}
    The computation for the other formulas is similar.
\end{proof}
\subsection{quantum Howe duality of $(\mathbf{U}^{c}_n,\mathbf{U}^{b}_m)$}

Let ${}_\mathbb{A}\mathbb{T}^{\mathfrak{c}\mathfrak{b}}_{n|m,d}$ be the Fock space consisting of $\mathrm{SO}_{2d+1}$-invariant functions on the product of two partial flag varieties of type B and $\mathcal{V}^{\mathfrak{b}\mathfrak{c}}_{m|n,d}$ be the subspace of a quantum coordinate coalgebra with  $\mathbb{A}$-basis $\{ \tilde{t}^{(A)}\mid A\in \Xi^{\mathfrak{b}\mathfrak{c}}_{m|n,d}\}$. Denote $\mathcal{V}^{\mathfrak{b}\mathfrak{c}}_{m|n}:=\bigoplus_{d=0}^{\infty}\mathcal{V}^{\mathfrak{b}\mathfrak{c}}_{m|n,d}$ and $_\mathbb{A}\mathbb{T}^{\mathfrak{c}\mathfrak{b}}_{n|m}:=\bigoplus_{d=0}^{\infty} {}_\mathbb{A}\mathbb{T}^{\mathfrak{c}\mathfrak{b}}_{n|m,d}  $. The $(\mathbf{U}^{c}_n,\mathbf{U}^{b}_m)$-bimodules structures of $\mathcal{V}^{\mathfrak{b}\mathfrak{c}}_{m|n}$ and ${}_\mathbb{A}\mathbb{T}^{\mathfrak{c}\mathfrak{b}}_{n|m}$ were introduced in \cite{LX22} and they are isomorphic as $(\mathbf{U}^{c}_n,\mathbf{U}^{b}_m)$-bimodules (see \cite[Theorem 5.7]{LX22}).

\begin{theorem}\label{Theorem: main thm}
     As $(\mathbf{U}^{c}_n,\mathbf{U}^{b}_m)$-bimodules,
    \begin{equation*}
        \begin{aligned} \mathscr{S}^{\mathfrak{b}\mathfrak{c}}_{m|n,d}\cong\mathcal{V}^{\mathfrak{b}\mathfrak{c}}_{m|n,d},\quad \mathscr{S}^{\mathfrak{b}\mathfrak{c}}_{m|n}\cong\mathcal{V}^{\mathfrak{b}\mathfrak{c}}_{m|n}:X^{[A]}\mapsto \tilde{t}^{(A)}.
        \end{aligned}
\end{equation*}
Consequently, $\mathscr{S}^{\mathfrak{b}\mathfrak{c}}_{m|n,d}\cong{}_\mathbb{A}\mathbb{T}^{\mathfrak{c}\mathfrak{b}}_{n|m,d}$ as $(\mathbf{U}^{c}_n,\mathbf{U}^{b}_m)$-bimodules and the actions
\begin{equation*}
    \begin{aligned}
        \mathbf U_n^{\mathfrak c}
        \curvearrowright
        \mathscr{S}_{m|n,d}^{\mathfrak b\mathfrak c}
        \curvearrowleft
        \mathbf U_m^{\mathfrak b},
        \qquad
        \mathfrak b,\mathfrak c\in\{\imath,\jmath\}
    \end{aligned}
\end{equation*}
satisfy the double centralizer property.
\end{theorem}
\begin{proof}
    Comparing the formulas in Proposition \ref{prop:the left action of Uj on Snij}, \ref{prop: the left action of Ui on Snji} and \ref{prop:the right action of Ub on Snbc} and \cite[Proposition 5.5]{LX22}, we obtain the desired isomorphisms. The double centralizer property follows from \cite[Theorem 4.3]{LX22}.
\end{proof}

\end{document}